\documentclass[12pt]{amsart}

\usepackage{amssymb,amsmath,amsbsy, amsthm}
\usepackage[english]{babel}
\usepackage{amscd}
\usepackage{pb-diagram}
\usepackage{pstricks,pst-node}
\usepackage{amsfonts}
\usepackage{latexsym}
\usepackage{pstricks,pst-node}
\usepackage{tikz}
\usetikzlibrary{positioning,shapes,shadows,arrows}
\usepackage{fullpage}

\usepackage{hyperref}

\newtheorem{theorem}{Theorem}[section]
\newtheorem{lemma}[theorem]{Lemma}
\newtheorem{proposition}[theorem]{Proposition}
\newtheorem{corollary}[theorem]{Corollary}
\theoremstyle{definition}
\newtheorem{definition}[theorem]{Definition}
\newtheorem{example}[theorem]{Example}
\newtheorem{problem}[theorem]{Problem}
\newtheorem{notation}[theorem]{Notation}
\theoremstyle{remark}
\newtheorem{remark}[theorem]{Remark}
\numberwithin{equation}{section}

\newtheorem{question}[theorem]{Question}

\def\int{\mathop{\rm int}\nolimits}
\def\cl{\mathop{\rm cl}\nolimits}
\def\ran{\mathop{\rm ran}\nolimits}



\begin{document}
\date{}
\title{Path connectedness, local path connectedness and contractibility of $\mathcal{S}_c(X)$}         
\author{Javier Camargo}
\address{Javier Camargo\\
	\newline
	Escuela de Matem\'aticas, Facultad de Ciencias, Universidad Industrial de
	Santander, Ciudad Universitaria, Carrera 27 Calle 9, Bucaramanga,
	Santander, A.A. 678, COLOMBIA.}
\email{jcamargo@saber.uis.edu.co}


\author{David Maya}
\address{David Maya\\
	\newline Universidad Aut\'onoma del Estado de M\'exico, Facultad de Ciencias, Instituto Literiario 100, Col. Centro, Toluca, CP 50000, M\'EXICO.
	}
\email{dmayae@outlook.com, dmayae@uaemex.mx}

\author{Patricia Pellicer-Covarrubias}
\address{Patricia Pellicer-Covarrubias\\
	\newline Departamento de Matem\'aticas, Facultad de Ciencias, Circuito ext. s/n, Ciudad Universitaria, C.P. 04510, CDMX, M\'EXICO.
}
\email{paty@ciencias.unam.mx}



\begin{abstract}
The hyperspace of all nontrivial convergent sequences in a Hausdorff space $X$ is denoted by  $\mathcal{S}_c(X)$. This hyperspace is endowed with the Vietoris topology. In connection with a question and a problem by Garc\'ia-Ferreira, Ortiz-Castillo and Rojas-Hern\'andez, concerning conditions under which $\mathcal S_c(X)$ is pathwise connected,  in the current paper we  study the latter property and the contractibility of $\mathcal{S}_c(X)$. We present necessary conditions on a space $X$ to obtain the path connectedness of $\mathcal{S}_c(X)$. We also provide some sufficient conditions on a space $X$ to obtain such path connectedness. Further, we characterize the local path connectedness of $\mathcal{S}_c(X)$ in terms of that of $X$. We prove the contractibility of $\mathcal{S}_c(X)$ for a class of spaces and, finally, we study the connectedness of Whitney blocks and Whitney levels for $\mathcal{S}_c(X)$.
\end{abstract}

\maketitle

Convergence of sequences is an important tool to determine topological properties in Hausdorff spaces. On the other hand, the study of hyperspaces provides information about the topological behavior of the original space and vice versa. In connection with both concepts, the hyperspace consisting of all nontrivial convergent sequences $\mathcal{S}_c(X)$, of a metric space $X$ without isolated points, was introduced in 2015 in \cite{Garcia2015}. 
Since then, there has been increasing interest in studying $\mathcal{S}_c(X)$ and several papers presenting relevant  properties of this hyperspace have been written:  \cite{GarciaC}, \cite{GarciaCP}, \cite{Illanes}, \cite{MayPelPic3}, \cite{MayaGScX} and \cite{MayaCF}.

In \cite{Garcia2015} the authors studied, among other properties, the path connectedness of $\mathcal{S}_c(X)$: they proved that $\mathcal{S}_c(X)$ is pathwise connected, whenever $X$ is either $\mathbb{R}$ or a connected space that is a finite union of copies of $[0,1]$ (\cite[Theorem~2.4 and Lemma~2.11]{Garcia2015}). Also, \cite[Example~2.8]{Garcia2015} showed a pathwise connected continuum $X$ such that $\mathcal{S}_c(X)$ is not pathwise connected. Further, the authors in \cite[Question~2.9]{Garcia2015} asked if $\mathcal{S}_c(X)$ is pathwise connected when $X$ is a dendroid; in \cite[Example~4.6]{MayaGScX}  this question was answered in the negative. They also asked whether the path connectedness of $\mathcal{S}_c(X)$ implies that of $X$ (\cite[Question~2.14]{Garcia2015}); this question was answered in the affirmative for infinite, non-discrete, Fr\'echet-Urysohn spaces  in \cite[Corollary~3.3]{GarciaC}. Despite the recent contributions to this study, the behavior of the path connectedness of $\mathcal{S}_c(X)$ is still not fully understood; in particular, necessary and sufficient conditions on a space $X$ in order for $\mathcal{S}_c(X)$ to be pathwise connected have still not been found. In connection with this situation, 
in \cite[Question~2.13]{Garcia2015} and  \cite[Problem~3.4]{GarciaC} the authors asked, respectively:

\begin{question} \label{preg:garcia}
What kind of topological properties must $X$ have when the hyperspace $\mathcal{S}_c(X)$ is pathwise connected? 
\end{question}

\begin{problem} \label{probl:garcia}
   Give conditions on $X$ under which $\mathcal{S}_c(X)$ must be pathwise connected.
\end{problem}

In the present paper, we address both problems and we study in more detail the behavior of the path connectedness of the hyperspace of nontrivial convergent sequences.  In Corollary~\ref{corolit875iii} and Theorem~\ref{theo:Ri} we give some partial answers to Question~\ref{preg:garcia}. Moreover, in Theorem~\ref{teor:Xcontr-implica-Sc(X)-path-conn} and Corollary~\ref{corol:equiv-entre-path-conn y local-path-conn-XySc(X)} we give  sufficient conditions on a space $X$ so that $\mathcal{S}_c(X)$ is pathwise connected, thus addressing Problem~\ref{probl:garcia}.

More precisely, our main results in this paper are the following: first, we show the path connectedness of $\mathcal{S}_c(X)$ for a contractible and first countable space $X$. 
Second, we prove that for a uniquely arcwise connected continuum $X$, if $\mathcal{S}_c(X)$ is pathwise connected, then $X$ must be uniformly pathwise connected (in fact, if $X$ is not uniformly pathwise connected, then $\mathcal S_c(X)$ has exactly $\mathfrak c$ path components). 
Third, we show that the path connectedness of $\mathcal{S}_c(X)$ implies the absence of $R^i$-sets in a continuum $X$ for all $i \in \{1,2,3\}$. 
Fourth, we establish the equivalence between  the local path connectedness of a first countable space $X$ and that of its hyperspace $\mathcal{S}_c(X)$. 
Fifth,  we determine the contractibility of $\mathcal{S}_c(X)$ for some spaces $X$. Finally, we study  the connectedness of Whitney blocks and Whitney levels for $\mathcal{S}_c(X)$.

We note that Corollary~\ref{corolit875iii}, Theorem~\ref{theo:Ri}, Theorem~\ref{teor:Xcontr-implica-Sc(X)-path-conn} and Corollary~\ref{corol:equiv-entre-path-conn y local-path-conn-XySc(X)} generalize significantly \cite[Example~2.8]{Garcia2015}, \cite[Example~4.6]{MayaGScX},  \cite[Corollary~2.5]{Garcia2015} and \cite[Lemma~2.11]{Garcia2015}, respectively.


\section{Preliminaries}

All topological notions and all set-theoretic notions whose definition is not included here should be understood as in \cite{Engelking1989} and \cite{Kunen1980}, respectively. 

The symbol $\omega$ denotes both, the first infinite ordinal and the first infinite cardinal. In particular, we consider all nonnegative integers as ordinals too; thus, $n\in\omega$ implies that $n=\{0,\dots,n-1\}$ and $\omega \setminus n=\{k\in\omega:k\geq n\}$. The successor of $\omega$ is the ordinal $\omega+1=\omega\cup\{\omega\}$ and
so the symbols $i\in\omega+1$, $i<\omega+1$ and $i\leq\omega$ all represent the same. The set $\omega \setminus \{0\}$ is denoted by $\mathbb{N}$. As usual, $\mathfrak c$ will be used to represent the cardinality of the real line, $\mathbb{R}$.

If $X$ is a set and $\kappa$ is a cardinal, $|X|$ will represent the cardinality of $X$ and $[X]^{<\kappa}$ denotes the family of all subsets of $X$ whose cardinality is $<\kappa$. In particular, $[X]^{<\omega}$ is the collection of all finite subsets of $X$ and $[X]^{<n+1}$ is the collection of all subsets of $X$ having at most $n$ elements, whenever $n \in \mathbb{N}$. The symbol $\mathcal P (X)$ will denote the power set of a set $X$.

For a function $f$, $\ran(f)$ will denote its range, and given a subset $A$ of the domain of $f$, the set $\{f(x) : x \in A\}$ is denoted by $f[A]$.

 

In this paper, {\sl space} means Hausdorff space. For a  space $X$, the symbol $\tau_X$ will denote the collection of all open subsets of $X$. Also, for a set $A\subseteq X$, we will use $\int_X A$ and $\cl_XA$ (or, if there is no risk of confusion, $\overline A$) to represent its interior in $X$ and its closure in $X$, respectively. 

A space $X$ is \emph{locally  connected (strongly locally  pathwise connected)} at a point  $p$ provided that  $X$ has a basis of open,   connected (pathwise connected) neighborhoods of $p$. Moreover,  $X$ is \emph{connected im kleinen (locally pathwise connected) at $p$} whenever $X$ has a basis of connected (pathwise connected) neighborhoods of $p$. Finally, $X$ is \emph{locally connected (strongly locally pathwise connected, connected im kleinen or locally pathwise connected)} provided that $X$ has such property at all of its points.

Our next lemma is known and easy to prove.

\begin{lemma} \label{lema:LAC-equiv} 
  The following conditions are equivalent for a space $X$:
     \begin{itemize}
        \item[ (i)] $X$ is strongly locally pathwise connected;
        \item[ (ii)] $X$ is locally pathwise connected; 
        \item[ (iii)] if $U$ is an open subset of $X$ and $K$ is a path component of $U$, then $K$ is open in $X$.
     \end{itemize}
\end{lemma}      

A space $ Z$ is \textit{contractible} provided that there exist  $p \in Z$ and a homotopy $h: Z \times [0,1] \to Z$ satisfying that $h(z,0) = z$ and $h(z,1) = p$ for each $z \in Z$. Such a homotopy $h$ is called a \emph{contraction}.

A {\sl continuum} is a compact, nonempty, connected, metric space. A \emph{dendrite} is a locally connected continuum that contains no simple closed curves.

A {\sl convergent sequence} in a topological space $X$ is a function $f$ from $\omega$ into $X$ for which there is $x \in X$ in such a way that for each $U \in \tau_X$ with $x \in U$ there exists $n \in \omega$ with $f[\omega \setminus n] \subseteq U$. In this case, we will say that either $f$ converges to $x$ or $x$ is the limit of $f$, and this fact will be denoted by either $\lim\limits_{n \to \infty} f(n) = x$ or $f(n) \to x$. We shall write $(f(n))_{n \in \omega}$ to refer to $f$. If $| \ran (f) | = \omega$, we say that $f$ is {\sl nontrivial}. In connection with this concept, in this paper, a subset $S$ of a space $X$ will be called a {\sl nontrivial convergent sequence in $X$} if $S$ is countably infinite and there is $x \in S$ in such a way that $S \setminus U \in [X]^{< \omega}$ for each $U \in \tau_X$ with $x \in U$. When this happens, the point $x$ is called {\sl the limit point of} $S$ and we will say that $S$ {\sl converges to} $x$ and write either $S \to x$ or $\lim S = x$. Throughout this paper, the reader will be able to identify from the context what is the intended meaning of {\sl nontrivial convergent sequence} in the discussion.


For a space $X$, define the hyperspaces: 
\begin{align*}
\mathcal{CL}(X) 	&= \{ A \subseteq X : A \text{ is closed in } X \text{ and } A \neq \emptyset\},	\\
\mathcal{K}(X) 		&= \{ A \in \mathcal{CL}(X) : A \text{ is compact} \},	\\
\mathcal{S}_c(X) 	&= \{ S \in \mathcal{K}(X) : S \text{ is a nontrivial convergent sequence in } X\}, 
\\
\mathcal{F}_n(X) &=\{A \in \mathcal{CL}(X) :  |A|\leq n\}, \text{ with } n\in\mathbb{N}, \\
\mathcal{F}(X) & = \{A \in \mathcal{CL}(X) : A \text{ is finite}\}=\bigcup_{n\in\mathbb{N}}\mathcal{F}_n(X) \text{ and } \\
\mathcal{C}(X) 		&= \{ A \in \mathcal{K}(X) : A \text{ is connected} \}.	
\end{align*}

Given a family $\mathcal U$ of subsets of $X$, we define 
    \[\langle\mathcal{U} \rangle = \left\{ A \in \mathcal{CL}(X): A\subseteq \bigcup\mathcal U\ 
    \wedge\ \forall\ U\in\mathcal U\ (A\cap U\neq\emptyset) \right\}.\] 
The {\sl Vietoris topology} is the topology on $\mathcal{CL}(X)$ generated by the base consisting of all sets of the form $\langle \mathcal{U} \rangle$, where $\mathcal{U} \in [\tau_X]^{< \omega}$ \cite[Proposition~2.1, p.~155]{Michael1951}. All the hyperspaces  described above will be considered as subspaces of $\mathcal{CL}(X)$. In particular, a base for the topology of $\mathcal{S}_c(X)$ consists of all sets of the form $\langle \mathcal U \rangle_c = \langle \mathcal U \rangle \cap \mathcal \mathcal{S}_c(X)$, where $\mathcal{U} \in [\tau_X]^{< \omega}$.

For a subset $V$ of a space $X$, let  $V^+_c = \langle \{V\} \rangle_c$ and $V^-_c = \langle \{X,V\} \rangle_c$. Thus, when $V$ is open (closed, resp.) in $X$, then  $V^+_c$ and $V^-_c$ are open (closed, resp.) in $\mathcal{S}_c(X)$. 

\begin{lemma} \label{lema:viet-conten}  \cite[2.3.1 of Lemma~2.3]{Michael1951}
  The following conditions are equivalent for a space  $X$ and $\mathcal U, \mathcal V \in [ \mathcal P(X) \setminus \{ \emptyset\}] ^{<\omega}$:
    \begin{enumerate}
        \item $\langle \mathcal V \rangle \subseteq \langle \mathcal U \rangle$; 
        \item $\bigcup \mathcal V \subseteq \bigcup \mathcal U$ and for each $U \in \mathcal U$, there exists $V \in \mathcal V$ such that $V \subseteq U$.
    \end{enumerate}
\end{lemma}

For   the sake of simplicity we will adopt the following convention: if $S$ is a nontrivial convergent sequence in a space $X$, we will say that $\{x_n : n \in \omega +1\}$ is an \textsl{adequate enumeration} of $S$ provided that $S = \{ x_n : n \in \omega +1\}$, $\lim S = x_\omega$ and $x_i \neq x_j$ whenever $i<j\leq \omega$.

A {\sl cellular family} in a topological space $X$ is a pairwise disjoint family of nonempty open subsets of $X$. The collection of all finite cellular families of $X$ is denoted by $\mathfrak{C}(X)$.

  \begin{theorem} \label{theo-celulares}
Let $X$ be a space and let $Y \in \mathcal K(X)$ be   zero-dimensional. Then $\{ \langle \mathcal V \rangle  : \mathcal  V  \in \mathfrak C (X)$ and $Y \in \langle \mathcal V \rangle \}$  is a local base for $\mathcal{CL}(X)$ at $Y$.
  \end{theorem}
  \begin{proof}
    Fix  $\mathcal U \in [\tau _X]^{<\omega}$ such that $Y \in \langle \mathcal U \rangle $. Let $F \in [Y]^{<\omega} \cap \langle \mathcal U \rangle $. For each $x \in F$, let $Q_x$ be a clopen subset of $Y$ such that $x \in Q_x \subseteq \bigcap \{U \in \mathcal U: x \in U\}$. Set $\mathcal Q_0 =  \{Q_x : x \in F\}$.  We may assume that the elements of $\mathcal Q _0$ are pairwise disjoint. 
    
If  $\bigcup \mathcal Q_0 \ne Y$, define $ \mathcal Q_1 = \mathcal Q_0 \cup \{ Y \setminus \bigcup \mathcal Q_0\}$; otherwise let $\mathcal Q_1 = \mathcal Q_0$. In either case,  $\mathcal Q_1$ is a finite  family of pairwise disjoint, nonempty, compact subspaces of $X$ whose union is $Y$.

Let $\mathcal V \in \mathfrak C (X)$ be such that: 

(i) $| \mathcal Q_1| = | \mathcal V|$,  

(ii) each element $Q$ of $\mathcal Q_1$ is contained in exactly one element $V_Q$ of $\mathcal V$,
 
  (iii) $Y \subseteq \bigcup \mathcal V \subseteq \bigcup \mathcal U$ and
  
   (iv) if $Q= Q_x$ for some $x \in F$, then $V_Q \subseteq  \bigcap \{U \in \mathcal U: x \in U\}$. 
   
\noindent   Using Lemma~\ref{lema:viet-conten} and the fact that $F \in \langle \mathcal U \rangle$, it follows easily that $Y \in \langle \mathcal V \rangle \subseteq \langle \mathcal U \rangle$. 
\end{proof}

  \begin{corollary}  \label{corol:UCXbase}
Let $X$ be a space and let $\mathcal H (X)$ be a subspace of $\mathcal K (X)$ with the property that each element of $\mathcal H (X)$ is zero-dimensional. Then $\{ \langle \mathcal V \rangle \cap \mathcal H (X) : \mathcal  V  \in \mathfrak C (X)\}$ is a base for $\mathcal H(X)$.
  \end{corollary}

The following result was proved in \cite[Proposition~3.2]{MayaGScX}.

\begin{proposition}\label{pro:UCXbase} 
For an arbitrary space $X$, $\{ \langle \mathcal{U} \rangle_c : \mathcal{U} \in \mathfrak{C}(X) \}$ is a base for $\mathcal{S}_c(X)$.
\end{proposition}

Observe that Corollary~\ref{corol:UCXbase} generalizes Proposition~\ref{pro:UCXbase}  and also it  has a shorter proof. We  note that  Proposition~\ref{pro:UCXbase} will be used constantly throughout the paper, although we may not always  mention it explicitly. Corollary~\ref{corol:UCXbase} will be used in Proposition~\ref{lemma2or3}.

The following are known results that will be useful in the rest of the paper.

\begin{lemma} \cite[Corollary~5.8.1, p.~169]{Michael1951} \label{lema:union-finita-cont}
    Let $X$ be a space. The function $\varphi : \mathcal{CL}(X)^n \to \mathcal{CL}(X)$ given by $\varphi(E_1, \dots , E_n) = \bigcup _{i=1}^n E_i$ is continuous.
\end{lemma}   

\begin{lemma} \label{lema:union-da-abierto} \cite[Lemma~4.8]{MayaGScX}
  Let $X$ be a space. If $\mathsf G$ is an open subset of $\mathcal \mathcal{S}_c(X)$, then $\bigcup \mathsf G$ is an open subset of $X$.
\end{lemma}

Let $X$ and $Y$ be spaces and let $f : X \to Y$ be a mapping. Define the \emph{induced mapping} $\mathcal{K}(f) : \mathcal{K}(X) \to \mathcal{K}(Y)$ by $$\mathcal{K}(f)(A) = f[A] \text{ for each } A \in \mathcal{K}(X).$$ 

\begin{theorem}   \cite[5.10.1 of Theorem~5.10, p.~170]{Michael1951} \label{teor:func-ind-cont}
  Given a mapping $f\colon X \to Y$, the induced mapping $\mathcal{K}(f) : \mathcal{K}(X) \to \mathcal{K}(Y)$ is continuous.
\end{theorem}

Using the proof of  \cite[Theorem~3.2]{GarciaC}  one can show the following result.

\begin{theorem} \label{teor:trayec-gamma-enX}   
	For a space $X$, let $\mathcal H(X)$ be a nonempty subset of $\mathcal K(X)$ with the property that each element of $\mathcal H(X)$ is zero-dimensional. If $\alpha:[0,1]\to\mathcal H(X) $ is a path, then for each $p\in\alpha(0)$ there is a path $\beta:[0,1]\to X$ in such a way that $\beta(t)\in\alpha(t)$, whenever $t\in [0,1]$, and $\beta(0)=p$. 
\end{theorem}

Our next result is proved for infinite, non-discrete, Fr\'echet-Urysohn spaces in \cite[Corollary~3.3]{GarciaC}. However, using the proof of \cite[Theorem~4.5]{MayaGScX} one can show it for arbitrary spaces.

\begin{theorem} \label{teor:Sc(X)-path-conn-implicaXtb} 
  If $X$ is a space such that $\mathcal \mathcal{S}_c(X)$ is nonempty and pathwise connected, then so is $X$.
\end{theorem}  

\begin{remark} \label{rem:c}
  If $X$ is a space such that $|X| = \mathfrak c$, then $|\mathcal S_c(X)| \le \mathfrak c ^\omega = \mathfrak c$. In particular, $\mathcal S_c(X)$ has at most $\mathfrak c$ path components.
\end{remark}

\begin{theorem} \cite[Corollary~31.6, p.~222]{Willard}  \label{lema:path-con-implica-arco-con}
  A $T_2$ space is pathwise connected if and only if it is arcwise connected.
\end{theorem}  

 \begin{remark} \label{remark:path-conn}
   If $Y$ is a pathwise connected space, then so is $\mathcal F_n(Y)$ for each $n \in \mathbb N$ \cite[(a), p.~877]{Borsuk-Ulam}. Therefore $\mathcal F (Y)$ is pathwise connected.
\end{remark}

We will use the following notation frequently.

\begin{notation}\label{unique-arcs}
   Let $X$ be uniquely arcwise connected space and let $x,y\in X$. If $x \neq y$, then the unique arc between $x$ and $y$ in $X$ will be denoted by $[x,y]$. If $x = y$, then  $[x,y]$ will denote the set $ \{x\}$.
\end{notation}


\section{Path connectedness of $\mathcal \mathcal{S}_c(X)$}

In this section we address Question~\ref{preg:garcia} and Problem~\ref{probl:garcia}. Namely, we obtain  necessary conditions on a space $X$ to obtain the path connectedness of $\mathcal \mathcal{S}_c(X)$; we also obtain a sufficient condition and present several examples.


\subsection{Auxiliary results}

\begin{lemma} \label{lema:mathcalK-path-conn}
   Let $X$ be a pathwise connected space. Define $\mathcal A = \{ A \subseteq X : A $ is an arc$\}$ and $\mathcal L = \{S \in \mathcal S _c(X) : $ there exists  $\mathcal B\in [\mathcal A]^{<\omega}$ such that $\bigcup \mathcal B$ is connected and $S \subseteq \bigcup \mathcal B\}$. Then:
   
   (i) $\mathcal L$ is pathwise connected,
   
   (ii)  if $F \in \mathcal F (X)$,  if $A \in \mathcal A$ and if $S \in \mathcal S _c(A)$, then $S \cup F \in \mathcal L$ and
   
   (iii) if $\mathcal{D}$ is the path component of $\mathcal{S}_c(X)$ such that $\mathcal{L}\subseteq \mathcal{D}$, then $\mathcal{D}$ is dense in $\mathcal{S}_c(X)$.
\end{lemma}
   \begin{proof}
     First we prove $(i)$, to this end fix $S_0, S_1 \in \mathcal L$. It is easy to show that there exists a finite subset $\mathcal B$ of $\mathcal A$    such that $\bigcup \mathcal B$ is connected and $S_0 \cup S_1 \subseteq \bigcup \mathcal B$. According to \cite[Lemma~2.11]{Garcia2015} the hyperspace $\mathcal S_c( \bigcup \mathcal B)$ is pathwise connected; since the latter is contained in $\mathcal L$ and   $S_0, S_1 \in \mathcal S_c( \bigcup \mathcal B)$, we conclude that $\mathcal L$ is pathwise connected. The proof of $(ii)$ is easy and left to the reader.
     
     Let $S\in \mathcal{S}_c(X)$ be arbitrary. Suppose that $\{s_i : i\in \omega+1\}$ is an adequate enumeration of $S$. Since $X$ is pathwise connected, there exist $A\in\mathcal{A}$ and $R\in \mathcal{S}_c(A)$, such that  $\lim R=s_{\omega}$. Notice that $R\in\mathcal{L}$. Let $\{r_i : i\in \omega+1\}$ be an adequate enumeration of $R$. Set $R_n=\{s_0,...,s_{n-1}\}\cup \{r_i : i\geq n\} \cup \{r_\omega\}$, for each $n\in\omega$. By (ii), $R_n\in \mathcal{L}$. Observe that $\lim_{i\to\infty}R_i=S$. Therefore, $S\in\mathrm{cl}_{\mathcal{S}_c(X)}(\mathcal{L})$ and $\mathcal{D}$ is dense in $\mathcal{S}_c(X)$.
   \end{proof}

\begin{theorem}\label{theo:ScXFX}
Let $X$ be a pathwise connected space, let $R \in \mathcal{S}_c(X)$ and let $m \in \mathbb{N}$. Then $R \in \mathcal{D}$ if and only if there exists a path $\lambda: [0,1] \to \mathcal{S}_c(X) \cup \mathcal{F}_m(X)$ such that $\lambda(0) = R$ and $\lambda^{-1}[ \mathcal{F}_m(X)] = \{1\}$.
\end{theorem}
\begin{proof}
Assume first that   $R\in\mathcal{D}$. Let $\mathcal A$ be as in Lemma~\ref{lema:mathcalK-path-conn} and fix $A \in \mathcal A$. It is not difficult to show that if $S\in \mathcal{S}_c(A)$,  then there exists a path $\alpha\colon [0,1] \to\mathcal{S}_c(A)\cup \mathcal{F}_1(A)$, such that $\alpha(0)=S$ and $\alpha^{-1}[ \mathcal{F}_1(A)] = \{1\}$. Moreover, there exists a path in $\mathcal S_c(X)$ from $R$ to $S$.  Thus,  there exists a path  $\lambda: [0,1] \to \mathcal{S}_c(X) \cup \mathcal{F}_m(X)$ fulfilling all our requirements.


Conversely, let $\lambda : [0,1] \to \mathcal{S}_c(X) \cup \mathcal{F}_m(X)$ be such that $\lambda(0) = R$ and $\lambda^{-1}[\mathcal{F}_m(X)] = \{1\}$.  Since $X$ is pathwise connected, we may fix $x \in R$ and an arc $L \subseteq X$, with end points $x $ and some distinct point $a$, and choose a homeomorphism $\beta: [0,1] \to L$ such that $\beta(0) = a$ and $\beta (1) = x$.       
   For each $n \in \mathbb N$  define $\alpha_n : [\frac{1}{n+1},\frac{1}{n}] \to   \mathcal S _c(X) $ by
   \[     \alpha_n(t) =  \lambda(1-t) \cup \{ \beta (t)\}  
             \cup   \left\{ \beta \left( \frac{1}{k} \right) : 
                      k \in \{1, \dots , n\} \right\}. \]
Observe that $\alpha_n$ is well defined, continuous and the equality $\alpha_n ( \frac{1}{n+1})  =   \alpha_{n+1} ( \frac{1}{n+1})$ holds for each $n \in \mathbb N$. Next define $\alpha : [0,1] \to \mathcal \mathcal{S}_c(X)$ as follows: 
           \begin{equation*}
               \alpha(t) = \left\{ \begin{array}{lll}
                   \lambda(1) \cup \{a\} \cup \left\{ \beta \left( \frac{1}{k}\right) :
                               k \in \mathbb N \right\} ,    & \textrm{if} & t =0; \\
                    \alpha_n(t), & \textrm{if} & t \in [\frac{1}{n+1},\frac{1}{n}], \   
                    \textrm{ for some } n \in \mathbb N.
          \end{array} \right.
       \end{equation*}   
Using  the fact that $\lambda(1) \in \mathcal{F}_m(X)$ one readily sees that $\alpha $ is well defined. Further, as a consequence of \cite[Theorem~18.3, p.~108]{Munkres} one can show that it is continuous in $(0,1]$. In order to check that $\alpha$  is continuous at $t=0$, fix $\mathcal W \in \mathfrak{C}(X)$ such that $ \alpha(0) \in \langle \mathcal W \rangle _c  $ (see Proposition~\ref{pro:UCXbase}). Define $\mathcal U = \{V \in \mathcal W : V \cap \lambda (1) \ne \emptyset\}$. Observe that $\lambda (1) \in \bigcap \{V_c^- : V \in \mathcal U\}$. Let $W \in \mathcal W$ be such that $a \in W$ and pick $N \in \mathbb N$ such that
    \begin{equation} \label{eq:le-maroc}
       \beta\left[ \left[0,\frac{1}{N}\right] \right] \subseteq W \ \ \text{  and } \ \ \lambda \left[  
        \left[1-\frac{1}{N},1  \right] \right] \subseteq \left( \bigcup \mathcal W \right)^+_c \cap \left( \bigcap \{V_c^- : V \in \mathcal U\} \right). 
    \end{equation}    
Fix $t \in [0, \frac{1}{N} \big)$. A routine argument shows that $\alpha(t) \subseteq \bigcup \mathcal W $. Now let $V \in \mathcal W$, then $V \cap \alpha(0) \ne \emptyset$. If $\beta (\frac{1}{k}) \in V$ for some $k \in \mathbb N$, the cellularity of $\mathcal W$ and the choice of $N$ guarantee  that $V \cap \alpha (t) \ne \emptyset$. Hence,  we may assume that $\lambda(1) \cap V \neq \emptyset$. This means that $V \in \mathcal{U}$. Thus, $\lambda(1-t) \in V^-_c$ and so, $\lambda(1-t) \cap V \neq \emptyset$. 
We have proved that $\alpha(t) \in \langle \mathcal W \rangle _c$.
Therefore, using the second part of Lemma~\ref{lema:mathcalK-path-conn}, we conclude that $\alpha$ is a path in $\mathcal \mathcal{S}_c(X)$ that joins $R$ with an element of $\mathcal D$.
\end{proof}

The authors do not know the answer to the following question.

\begin{question}
	Let $X$ be a pathwise connected space and let $R\in\mathcal{S}_c(X)$. Suppose that there exists a path $\lambda\colon [0,1]\to \mathcal{S}_c(X)\cup\mathcal{F}(X)$ such that $\lambda(0)=R$ and $\lambda[[0,1]]\cap \mathcal{F}(X)\neq\emptyset$. Does it follow that $R\in\mathcal{D}$?
\end{question}


\subsection{A sufficient condition for the path connectedness of $\mathcal S_c(X)$}

In this subsection we address Problem~\ref{probl:garcia}, by showing that contractible,  first countable  spaces $X$ have pathwise connected hyperspace  $\mathcal S_c(X)$. We also show that the converse is not true.

   \begin{lemma} \label{lema:vecindades-locales-1}
    Let $X$ be a pathwise connected space,   let $y_\omega \in X$ and assume that $\{V_n : n \in \mathbb N\}$ is a countable local base at $y_\omega$ with $V_1 = X$. For each $n \in \mathbb N$ we will assume that $V_{n+1} \subseteq V_n$ and we will denote by $C_n$ the path component of $V_n$ that contains $y_\omega$.
Let  $ \{ y_n : n \in (\omega +1 ) \setminus \{0\}\}$ and $ \{ x_n : n \in (\omega +1 ) \setminus \{0\}  \}$ be adequate enumerations of two sequences $S_y$ and $S_x$, respectively, such that $x_\omega = y_\omega$. Further, assume that
  \[ \textrm{(*) for each $n \in \mathbb N$ there exists $m_n \in \mathbb N$ such that $ y_r , x_r \in C_n$ for all $r \ge m_n$.}\]
Then  there exists a path in $  \mathcal S_c (X) $ from $S_y$ to $S_x $.
\end{lemma}
   \begin{proof}
       For each $n \in \mathbb{N}$  define $M(n) = \max\{k \in \mathbb{N} : x_n , y_n \in C_k\}$, then there exists a path $\alpha _n : [\frac{1}{n+1} , \frac{1}{n}] \to V_{M(n)}$ such that  $\alpha _n(\frac{1}{n+1}) = y_n$ and $\alpha_n( \frac{1}{n} ) = x_n$. Note that $M(n)$ tends to infinity as $n$ tends to infinity.
    
   Next, for each $n \in \mathbb N$  define $g_n : [\frac{1}{n+1},\frac{1}{n}] \to   \mathcal S _c(X) $ by
   \[     g_n(t) =      \{x_m : m \ge n+1 \} \cup \{ \alpha_n(t) \}   \cup  \{y_m : 1 \le m < n \}  . \]
Observe that $g_n$ is well defined and continuous.  Also note that
  \begin{equation} \label{eq:g-n-rollote-simplif}
     g_n \left( \frac{1}{n+1}\right)  =  \{x_m : m \ge n+1 \} \cup  \{y_m : 1 \le m \le n \}  =  g_{n+1} \left( \frac{1}{n+1}\right). 
  \end{equation}                           
Now define $\gamma :[0, 1] \to   \mathcal S _c(X) $ by
     \begin{equation*}  
         \gamma(t) = \left\{ \begin{array}{lll}
               g_n(t),      & \textrm{if} & t \in [\frac{1}{n+1},\frac{1}{n}] \textrm{ for some } n \in \mathbb N; \\
            \{ y_n : n \in \omega +1  \}  , & \textrm{if} & t =0.
          \end{array} \right.
       \end{equation*}
It is not difficult to verify that $\gamma(0) = S_y  $ and $\gamma (1) = g_1( 1) = S_x $.  
          
According to (\ref{eq:g-n-rollote-simplif}) we have that $\gamma$ is well-defined, and using  \cite[Theorem~18.3, p.~108]{Munkres} one can show that it is continuous in $(0,1]$. In order to check that $\gamma$  is continuous at $t=0$, fix a finite cellular family $\mathcal W$ such that $S_y  = \gamma(0) \in \langle \mathcal W \rangle _c  $ (see Proposition~\ref{pro:UCXbase}).
Let $W \in \mathcal W$ such that $y_\omega \in W$ and take $N \in \mathbb N$ such that $ V_{M(n)} \subseteq W$ whenever $n \ge N$. 
Observe that $x_n , y_n \in V_{M(n)} \subseteq W$ for all $n \ge N$.
Hence, if $t \in [0, \frac{1}{N})$ a routine argument shows that $\gamma(t) \in \langle \mathcal W \rangle _c$ and, thus, $\gamma$ is a path in $\mathcal \mathcal{S}_c(X)$  from $S_y $ to $S_x  $.
   \end{proof}

\begin{corollary} \label{corol:loc-path-con-1onum}
	Let $X$ be a pathwise connected space and let $p\in X$. Assume that  $X$ is first countable at $p$. If $X$ is locally pathwise connected   at $p$ and  if $R, S  \in \mathcal \mathcal{S}_c(X)$ are such that $R\to p$ and $S\to p$, then there exists a path in $\mathcal \mathcal{S}_c(X)$ from $R$ to $S$.
\end{corollary}

\begin{theorem} \label{teor:Xcontr-implica-Sc(X)-path-conn}
  If $X$ is a contractible, first countable space, then $\mathcal \mathcal{S}_c(X)$ is   pathwise connected.
\end{theorem}  
   \begin{proof}
      Let $\mathcal A$ and $\mathcal L$ be as in Lemma~\ref{lema:mathcalK-path-conn}. Fix $S_0 \in \mathcal S _c(X)$. We will show that $S_0$ may be joined by a path in $\mathcal S _c(X)$ with some element of $\mathcal L$.  In this proof we will assume that $ \{ x_n : n \in \omega +1\}$ is an adequate enumeration of $S_0$.
      
      By assumption there exists a homotopy $H : X \times [0,1] \to X$ such that $H(x,0) = x_\omega$ and $H(x,1) = x$ for all $x \in X$. Observe that the sequence $\big( H(x_n, t) \big)_{n \in \omega}$ converges to $H(x_\omega, t)$ for each $t \in [0,1]$; thus,
 $H \big[ S_0 \times \{t\} \big] \in \mathcal F(X) \cup \mathcal S _c(X)$ for each $t \in [0,1]$.
      Define $T_0 = \{ t \in [0,1] : H \big[ S_0 \times \{t\} \big] \in \mathcal F(X) \}$ and $t_0 = \sup T_0$.  We analyze two cases. \medskip
      
    \noindent  \textbf{Case 1}. $t_0 \in T_0$. \medskip
      
Set $m = \big| H[S_0 \times \{t_0\}] \big|$. The path $\lambda : [0,1] \to \mathcal{S}_c(X) \cup \mathcal{F}_m(X)$ defined by $\lambda(t) = H[S_0 \times \{(1-t) + t_0 \cdot t\}]$ satisfies all conditions in Theorem~\ref{theo:ScXFX}. So, we can conclude that $S_0 \in \mathcal{D}$. \medskip

      

     \noindent  \textbf{Case 2}.  $t_0 \notin T_0$. \medskip
   
In this case  $H \big[ S_0 \times \{t_0\} \big] \in   \mathcal S _c(X)$ and the point  $z= H( x_\omega, t_0)$ is the limit point of $H \big[ S_0 \times \{t_0\} \big]$. Fix a countable local base $\{V_n : n \in \mathbb N\}$ of $z$ with the property that $V_{n+1} \subseteq V_n$ for each $n$. Also, let $C_n$ be the path component of $V_n$ that contains $z$, for each $n \in \mathbb N$. \medskip
   
 \noindent  \textbf{Claim}. For each $n \in \mathbb N$ there exists $m_n \in \mathbb N$ such that $H ( x_r , t_0) \in C_n$ whenever $r \ge m_n$. \medskip
   
   Let $n \in \mathbb N$. Fix a finite cellular family $\mathcal U _n$ such that $H \big[ S_0 \times \{t_0\} \big] \in \langle \mathcal U _n \rangle _c$ and if $U _n$ is the element of $\mathcal U _n$ that contains $z$, then $U _n \subseteq V_n$. Fix  $t_n \in T_0$ such that $H \big[ S_0 \times [t_n ,t_0 ] \big] \subseteq \langle \mathcal U_n \rangle $.
   
   Since $H \big[ S_0 \times \{t_n\} \big]$ is finite and nonempty there exist $y \in H \big[ S_0 \times \{t_n\} \big]$ and an infinite subset $J$ of $\mathbb N$ such that $ H ( x_r, t_n  ) = y $   whenever  $r \in J$; hence $H(x_\omega, t_n) = y$. Thus, using again the continuity of $H$ there exists $m_n \in \mathbb N$ such that 
    \begin{equation} \label{eq:lahomotophaciay}
        H ( x_r , t_n) =y \textrm{  whenever } r \ge m_n.    
    \end{equation}
Now, observe that 
    \begin{equation} \label{eq:union-homot-dentrounion-vietor}
       \bigcup H \big[ S_0 \times [t_n ,t_0 ] \big] \subseteq \bigcup \mathcal U _n.
    \end{equation}    
The cellularity of $\mathcal U _n$ and the path connectedness of $H \big[ \{x_\omega\} \times [t_n ,t_0 ] \big]$ imply that $y \in U _n$; therefore, since $U _n \subseteq V_n$ we deduce  that $y \in C_n$. Similarly, using (\ref{eq:lahomotophaciay}) and (\ref{eq:union-homot-dentrounion-vietor}) we conclude that $H( x_r, t_0) \in C_n$ whenever $r \ge m_n$ and the claim is proved.
   
   \vspace{.2cm}      

   Let $A \subseteq X$ be an arc that contains $z$ and let $S_1 \in \mathcal S_c(A)$ be such that $\lim S_1 = z$. Note that $C_n$ contains a tail of $S_1$ for each $n \in \mathbb N$.
    Since $\lim S_1 = z = \lim H \big( S_0 \times \{t_0\} \big)$, by Lemma~\ref{lema:vecindades-locales-1} there exists a path in $  \mathcal S _c(X)$ from $S_1$ to $H \big[ S_0 \times \{t_0\} \big]$. According to the choice of $t_0$ we know that $H \big[ S_0 \times [t_0 , 1 ] \big] $ is  a path in $\mathcal \mathcal{S}_c(X)$ from $H \big[ S_0 \times \{t_0\} \big]$ to $S_0$; therefore,  there exists a path in $\mathcal \mathcal{S}_c(X)$ from $S_1$ to $S_0$.  The   result follows from the fact that  $S_1 \in \mathcal L$.  
   \end{proof}

\begin{question}
	Is it true that $\mathcal{S}_c(X)$ is pathwise connected whenever $X$ is contractible? In other words, can the  assumption on  first countability be removed in Theorem~\ref{teor:Xcontr-implica-Sc(X)-path-conn}?
\end{question}

\vspace{.2cm}      

Recall that a continuum $X$ is {\it unicoherent} if  each pair of subcontinua of $X$ whose union is $X$, has connected intersection. 
A continuum is {\it hereditarily unicoherent} if all its subcontinua are unicoherent. A continuum is a 
{\it dendroid} if it is arcwise connected and hereditarily unicoherent.
A {\it fan} is a dendroid that has exactly one ramification point.

According to Theorem~\ref{teor:Xcontr-implica-Sc(X)-path-conn},  the  contractibility of a space $X$ is a sufficient condition to obtain the path connectedness of $\mathcal \mathcal{S}_c(X)$, among first countable spaces. It is however not a necessary condition,  as the 1-sphere $S^1$ shows (see (i) of Lemma~\ref{lema:mathcalK-path-conn}). 
This situation 
is the same even for the class of dendroids, as we prove in our next example.
We will use the following notation: 

\begin{notation}\label{broken}
 If $p_1,\dotsc,p_n$ are points in $\mathbb{R}^2$, we denote by $p_1\cdots p_n$ the broken line with  vertices $p_1,\dotsc,p_n$. In particular $p_1p_2$ denotes the segment that joins $p_1$ and $p_2$. 
\end{notation}

\begin{example}
    There exists a noncontractible dendroid $X$ (even a fan) such that $\mathcal \mathcal{S}_c(X)$ is pathwise connected.
\end{example}    
   \begin{proof}
      For each $n \in \mathbb N$ set $\textstyle  L_n = (0,0)(1, \frac{1}{n})(1+\frac{1}{n},0)(1, -\frac{1}{n})(0, -\frac{1}{n})$. Define $L_0 = (0,0)(1,0)$ and $X = \bigcup \{L_n : n \in \omega\}$. Then  $X$ is  a noncontractible fan \cite[p.~95]{JJC-Eberhart}.

Consider two subsets of $X$ given by $ X^+ = \{ (x,y) \in X : y \ge 0\}$   and $X^- = X \setminus X^+$.
Set $  Q = \{ (a,F) : a \in (0,1] \text{ and } F \text{ is a finite subset of } \mathbb N\}$. Further, for each $(a, F) \in  Q$ define 
   \[ X(a,F) = X^+ \cup \{ (x,y) \in X^- : x \ge a\} \cup \bigcup \{L_n : n \in F\}.\]
Note that each $X(a,F)$ is contractible, thus $\mathcal S_c \big( X(a,F) \big)$ is  pathwise connected  (Theorem~\ref{teor:Xcontr-implica-Sc(X)-path-conn}). 
Hence, if we define $\mathcal Y = \bigcup \{ \mathcal S_c(X(a,F)) : (a,F) \in Q\}$, then $\mathcal Y$ is a pathwise connected subspace of $\mathcal \mathcal{S}_c(X)$.

Fix $S_1 \in \mathcal \mathcal{S}_c(X)$, we will show that there exists a path, in $\mathcal \mathcal{S}_c(X)$, from $S_1$ to some element of $\mathcal Y$.
Observe that if $\lim S_1 \ne (0,0)$, then $S_1 \subseteq X(a,F)$ for some $(a,F) \in  Q$. Thus, $S_1 \in \mathcal Y$. Also, if $\lim S_1 = (0,0)$ and $S_1 \cap X^-$ is finite, it follows easily that $S_1 \in \mathcal Y$. Hence, from now on we will assume that $\lim S_1 = (0,0)$ and that $S_1 \cap X^-$ is infinite.

Take a continuous function $H : X^+ \times [-1,0] \to X^+$ such that
   \begin{equation} \label{eq:condicdelaH}
       H(z,-1) = z, \ H(z,0) = (0,0) \ \text{ and } \ H\big( (0,0),t \big) = (0,0)
   \end{equation}
whenever  $z \in X^+$ and  $t \in [-1,0]$ . 
Consider the mapping $\beta:  ( X^- \cup \{(0,0)\}) \times [0,1] \to  X$ given by $\beta \big( (x,y),t\big) = \big(x(1-t) + t, y\big)$ if $x \le 1$ and by the identity otherwise.
Finally, define $\gamma :[-1,1] \to \mathcal \mathcal{S}_c(X)$ by 
        \begin{equation*}
         \gamma(t) = \left\{ \begin{array}{lll}
            \mathcal K(H) \big( (S_1 \cap X^+) \times \{t\} \big) \cup (S_1 \cap X^-),
                 & \textrm{if} & t \in [-1,0]; \\
            \mathcal K(\beta) \big( \big( (S_1 \cap X^-) \cup \{(0,0)\}\big) \times \{t\}\big), 
                 & \textrm{if} & t \in [0,1].
          \end{array} \right.
       \end{equation*}   
Using (\ref{eq:condicdelaH}), the continuity of $H$ and $\beta$ and our assumptions on $S_1$, one can show that $\gamma$ is well defined. Moreover,   Theorem~\ref{teor:func-ind-cont} and Lemma~\ref{lema:union-finita-cont} imply that $\gamma$ is  continuous.
Therefore, $\gamma$ is a path from $S_1$ to some element of $\mathcal Y$, as desired.
   \end{proof}


\subsection{Necessary conditions for the path connectedness of $\mathcal S_c(X)$}

 In this subsection, we present necessary conditions on a continuum $X$ to obtain the path connectedness of $\mathcal{S}_c(X)$ (Corollary~\ref{corolit875iii} and Theorem~\ref{theo:Ri}), thus giving partial answers to Question~\ref{preg:garcia}.

We start giving some definitions and auxiliary results.
We point out that in the main results of  this subsection, every space $X$ will be a continuum.

Given $\epsilon>0$ and a metric space $X$, set $\mathcal C_\epsilon(X) = \{ A \in \mathcal  C(X) : \text{diam} (A) < \epsilon \}$.

\begin{definition}\label{defUAC}
	Let $X$ be a metric space and let $\mathcal{A} \subseteq \mathcal C (X)$. We say that \textit{$\mathcal{A}$ has property S uniformly} (abbreviated: \textit{$\mathcal{A}$ has PSU}), provided that for each $\epsilon>0$, there exists $k\in\mathbb N$ such that if $A\in\mathcal{A}$, then there exists $\mathcal{B} \in [\mathcal C_\epsilon(X)]^{<k+1}$ satisfying that $A = \bigcup \mathcal{B}$.
\end{definition}

\begin{remark}\label{lem:UACsub}
Let $X$ be a metric space. If $\mathcal{F} \subseteq \mathcal{E} \subseteq \mathcal C (X)$ and $\mathcal{E}$ has PSU, then $\mathcal{F}$ has PSU.
\end{remark}

Let $\mathcal{A}$ be a family of subsets  of a set $Z$. For each $D \subseteq Z$ define $\mathcal A_D = \{A \in \mathcal A : D \subseteq A\}$.
A \textit{refinement} of the family  $\mathcal{A}$   is a family of subsets $\mathcal{D}$ of $Z$ such that $\mathcal{A}_D \ne \emptyset$ for each  $D \in \mathcal{D}$.

For a metric space $X$, a refinement $\mathcal{D} \subseteq \mathcal C (X)$ of $\mathcal{A} \subseteq \mathcal C(X)$ is called a \textit{strong refinement} provided that $\{D \cap B : D \in \mathcal D, B \in \mathcal C(A) $ and $ A \in \mathcal A_D\} \subseteq \mathcal C(X) \cup \{\emptyset\}$.

\begin{remark}\label{rem:sref}
Let $X$ be a  metric space. Each refinement of a family $\mathcal A$ of dendrites of $X$ is a strong refinement of $\mathcal A$.
\end{remark}

\begin{lemma}\label{lem:srUAC}
Let $X$ be a metric space. If $\mathcal{A} \subseteq \mathcal C(X)$ has PSU and $\mathcal{D}$ is a strong refinement of $\mathcal{A}$, then $\mathcal{D}$ has PSU.
\end{lemma}
\begin{proof}
Let $\epsilon > 0$. From our assumption that $\mathcal{A}$ has PSU, there exists $k \in \mathbb N$ fulfilling that for each $A \in \mathcal{A}$ there exists $\mathcal{B} \in [\mathcal C_\epsilon(X)]^{<k+1}$ such that $A = \bigcup \mathcal{B}$.

Now, let $D \in \mathcal{D}$ be arbitrary. Since $\mathcal{D}$ is a refinement of $\mathcal{A}$,   there exists $A_D \in \mathcal{A}_D$. So, we may take $\mathcal{B} \in [\mathcal C_\epsilon(X)]^{< k+1}$ satisfying that $A_D = \bigcup \mathcal{B}$. Define $\mathcal{E} = \{ D \cap B : B \in \mathcal{B}\} \setminus \{ \emptyset\}$. Observe that $D = \bigcup \mathcal{E}$. The assumption that $\mathcal{D}$ is a strong refinement of $\mathcal{A}$ guarantees that $\mathcal{E} \in [\mathcal C_\epsilon(X)]^{< k+1}$. This proves that $\mathcal{D}$ has PSU.
\end{proof}

\begin{proposition}\label{lemma2or3}
	Let $X$ be a metric space and let $\mathcal{H}(X)$ be a nonempty subset of $\mathcal{K}(X)$ with the property that each element of $\mathcal{H}(X)$ is zero-dimensional.
	Suppose that there exists a path $f\colon [0,1]\to \mathcal{H}(X)$.  If for each $x \in f(0)$, there exists a path $g_x\colon [0,1]\to X$ such that $g_x(0)=x$ and $g_x(t)\in f(t)$ for each $t\in [0,1]$, then $\{g_x[[0,1]] : x \in f(0)\}$ has PSU.
\end{proposition}

\begin{proof}
	Let $\epsilon>0$. We show that there exists $k \in \mathbb N$ such that for each $x \in f(0)$ there exists $\mathcal{B}_x \in [\mathcal C_\epsilon(X)]^{<k+1}$ such that $g_x[[0,1]] = \bigcup \mathcal{B}_x$.
 	
	For each $t\in [0,1]$, let $\mathcal{U}_t \in \mathfrak{C}(X)$ be such that $f(t)\in \langle \mathcal{U}_t\rangle$ and $\mathrm{diam}(U)<\epsilon$ for each $U \in \mathcal{U}_t$ (see Corollary~\ref{corol:UCXbase}). It is clear that $f[[0,1]]\subseteq \bigcup\{ \langle \mathcal{U}_t\rangle : t\in [0,1]\}$. Thus, since $f[[0,1]]$ is compact, by \cite[Theorem~22.5]{Willard}, there exists a Lebesgue number $\delta>0$ for the open cover $\{ \langle \mathcal{U}_t\rangle : t\in [0,1]\}$.
	Since $f$ is uniformly continuous, there exists $\mathcal{J} \in [\mathcal C([0,1])]^{< \omega}$ such that $[0,1] = \bigcup \mathcal{J}$ and $\mathrm{diam}(f[J])<\delta$ for each $J \in \mathcal{J}$. Hence, for each $J \in \mathcal{J}$, there is $s_J\in [0,1]$ such that $f[J]\subseteq \langle\mathcal{U}_{s_J}\rangle$. Set $k = |\mathcal{J}|$.
	
	Now, for each $x \in f(0)$, define $\mathcal{B}_x = \{ g_x[J] : J \in \mathcal{J} \}$. Let us show that each $\mathcal{B}_x$ satisfies the required properties. Observe that $\mathcal B_x \in [\mathcal C(X)] ^{< k+1}$ and $g_x[[0,1]] = \bigcup \mathcal B_x$. To end the proof, it suffices to prove that $\bigcup \{\mathcal{B}_x : x \in f(0) \} \subseteq \mathcal C_\epsilon(X)$.
	
	Let $x \in f(0)$ be arbitrary. Take $J \in \mathcal{J}$. Then $g_x[J]\subseteq \bigcup f[J]\subseteq \bigcup \mathcal{U}_{s_J}$. Thus, since $g_x[J]$ is connected and $\mathcal{U}_{s_J} \in \mathfrak{C}(X)$, we have that $g_x[J]\subseteq U$ for some $U \in \mathcal{U}_{s_J}$. Therefore, $\mathrm{diam}(g_x[J])<\epsilon$. 
  \end{proof}

\begin{lemma}\label{lema08}
	Let $X$ be a metric space and let $\mathcal{A}\subseteq \mathcal C(X)$. If $\mathcal{A}$ does not have PSU, then there exists $\mathcal{D}\subseteq \mathcal{A}$ such that:
	\begin{enumerate}
		\item $\mathcal{D}$ does not have PSU, 
		\item if $\mathcal{C}\subseteq \mathcal{D}$ and $\mathcal{C}$ is infinite, then $\mathcal{C}$ does not have PSU and
		\item $\mathcal{D}$ is countable.
	\end{enumerate}
\end{lemma}

\begin{proof}
	Suppose that $\mathcal{A}$ does not have PSU. Then, there is $\epsilon>0$ such that, for each $k\in\mathbb N$, there exists $A_k\in\mathcal{A}$, with the property that:
	\begin{center}
		if $\mathcal{B}\in [\mathcal C(X)]^{<k+1}$ and $A_k=\bigcup\mathcal{B}$, then $B\notin \mathcal C_{\epsilon}(X)$, for some $B\in\mathcal{B}$.
	\end{center}	
	Let $\mathcal{D}=\{A_n : n\in\mathbb N\}$. It is not difficult to see that $\mathcal{D}$ satisfies \textit{1,2} and \textit{3}.
\end{proof}

A uniquely arcwise connected continuum $X$ is called \textit{uniformly arcwise connected} provided that for each $\epsilon>0$, there is a positive integer $n$ such that  any arc  in $X$ contains $n$ points that cut the arc into subarcs of diameter less than $\epsilon$ (see \cite{CharatonikUAC}). The following theorem is not difficult to prove.

\begin{theorem}\label{theoEq1}
Let $X$ be a uniquely arcwise connected continuum. The following conditions are equivalent:
\begin{enumerate}
\item $X$ is uniformly arcwise connected;
\item the family of all arcs in $X$ has PSU;
\item for each $z\in X$, the family of all arcs in $X$ with $z$ as an end point, has PSU;
\item for each $z\in X$, there exists a family  $\mathcal{A}$ of arcs having PSU, such that $z$ is an end point of each element of $\mathcal{A}$ and $X=\bigcup\mathcal{A}$.
\end{enumerate}
\end{theorem}

In our next result we will use Notation~\ref{unique-arcs}.

\begin{theorem}  \label{theoit875i}
	Let $X$ be a uniquely arcwise connected continuum. If $X$ is not uniformly
	arcwise connected, then $\mathcal{S}_c(X)$ has exactly $\mathfrak c$ arc components.
\end{theorem}

\begin{proof}
	Let $X$ be a uniquely arcwise connected continuum which is not uniformly arcwise connected. As a consequence of Lemma~\ref{lema08}, Theorem~\ref{theoEq1}, the compactness of $X$ and Remark~\ref{lem:UACsub} there exist  $z\in X$ and $\{x_i : i\in\omega+1\}\in \mathcal{S}_c(X)$  such that $\mathcal{D}=\{[x_i,z] : i\in\omega+1\}$ does not have PSU.
	
\noindent  \textbf{Claim 1}. 		For each $x\in X$, the set $\mathcal{D}_x=\{[x_i,z]\cup [z,x] : i\in\omega+1\}$ does not have PSU.

\vspace{.1cm}

	Since $\mathcal{D}$ is a strong refinement of $\mathcal{D}_x$ for each $x\in X$, the claim follows from Lemma~\ref{lem:srUAC}.

\vspace{.3cm}

\noindent  \textbf{Claim 2}.  For each $x\in X$, the set $\mathcal{E}_x=\{[x_i,x] : i\in\omega+1\}$ does not have PSU.

\vspace{.1cm}
	
	Suppose that $\mathcal{E}_x$ has PSU for some $x \in X$. Let $\epsilon>0$. Let $\mathcal{A}\in [\mathcal C_{\epsilon}(X)]^{<\omega}$ such that $[z,x]=\bigcup\mathcal{A}$. Since $\mathcal{E}_x$ has PSU, there exists $K\in\mathbb N$, such that for each $i\in\omega+1$, there is $\mathcal{B}_i\in [\mathcal C_{\epsilon}(X)]^{<K+1}$, with $[x_i,x]=\bigcup\mathcal{B}_i$. Let $l=|\mathcal{A}|$. Thus, if $\mathcal{C}_i=\mathcal{A}\cup\mathcal{B}_i$, then $\mathcal{C}_i\in [\mathcal C_{\epsilon}(X)]^{<K+l+1}$ and $[x_i,x]\cup [z,x]=\bigcup\mathcal{C}_i$. Note that $[x_i,x]\cup [z,x]=[x_i,z]\cup [z,x]$, for each $i\in\omega+1$. Therefore, $\mathcal{D}_x$ has PSU, contradicting Claim~1.
	
\vspace{.3cm}


\noindent  \textbf{Claim 3}.  		Let $x,y\in X$ and let $\epsilon>0$. If 
		   $$\exists n\in\mathbb N, \forall i\in\omega+1, \exists \mathcal{B}_i\in [\mathcal C_{\epsilon}(X)]^{<n+1} ([x_i,y]=\bigcup\mathcal{B}_i),$$ 
		then
		    $$\exists m\in\mathbb N, \forall i\in\omega+1, \exists \mathcal{L}_i\in [\mathcal C_{\epsilon}(X)]^{<m+1} ([x_i,x]=\bigcup\mathcal{L}_i).$$

\vspace{.1cm}
		
	Suppose that there exists $n\in\mathbb N$ such that for each $i\in\omega+1$, there is $\mathcal{B}_i\in [\mathcal C_{\epsilon}(X)]^{<n+1}$, with $[x_i,y]=\bigcup\mathcal{B}_i$. Let $\mathcal{A}\in [\mathcal C_{\epsilon}(X)]^{<\omega}$ such that $[x,y]=\bigcup\mathcal{A}$. Let $l=|\mathcal{A}|$. 
	For each $i \in \omega +1$ define
	  $$\mathcal{L}_i=\{[x_i,x]\cap D : D\in \mathcal{A}\cup\mathcal{B}_i\}\setminus \{\emptyset\}.$$ 
	It is clear that $\mathcal{L}_i\in [\mathcal C_{\epsilon}(X)]^{<l+n+1}$ and $[x_i,x]=\bigcup\mathcal{L}_i$. The proof of  Claim~3  is completed.

	
	\vspace{0.5cm}
	
	Since $\mathcal{D}$ does not have PSU, we may take $\epsilon>0$ such that, for each $n\in\mathbb N$, there is $i\in\omega+1$, satisfing that: 
	\begin{center}
		if $\mathcal{B}\in [\mathcal C(X)]^{<n+1}$ and $[x_i,z]=\bigcup\mathcal{B}$, then $B\notin \mathcal C_{\epsilon}(X)$, for some $B\in\mathcal{B}$.
	\end{center}
	
	For each $i\in\mathbb N$, let 
	    $$\mathcal{C}_i=\{[x_i,x_j] : j\in \omega\setminus (i+1)\}.$$ 
	Notice that, by Claim~2  and Lemma~\ref{lema08}, we may suppose that $\mathcal{C}_i$ does not have PSU, and by Claim~3, for each $n\in\mathbb N$, there exists $j_0\in \omega\setminus (i+1)$, such that:
	\begin{center}
		if $\mathcal{B}\in [\mathcal C(X)]^{<n+1}$ and $[x_i,x_{j_0}]=\bigcup\mathcal{B}$, then $B\notin \mathcal  C_{\epsilon}(X)$, for some $B\in\mathcal{B}$. $(\dag)$
	\end{center}

	
\noindent  \textbf{Claim 4}. 		There exists $(x_{n_k})_k\subseteq (x_n)_n$ such that for each $N\in\mathbb N$:
		\begin{center}
			if $i<N$, $\mathcal{B}\in [\mathcal C(X)]^{<N+1}$ and $[x_{n_i},x_{n_N}]=\bigcup\mathcal{B}$, then $B\notin \mathcal C_{\epsilon}(X)$, for some $B\in\mathcal{B}$.
		\end{center}

\vspace{.1cm}

	Let $n_1=1$, by $(\dag)$, there exists $n_2>1$ such that $[x_{n_1},x_{n_2}]\notin \mathcal C_{\epsilon}(X)$. Furthermore, by Lemma~\ref{lema08}, we may assume that there exists $l>n_2$ such that:
	\begin{center}
		if $i\geq l$, $\mathcal{B}\in [\mathcal C(X)]^{<3}$ and $[x_{n_1},x_i]=\bigcup\mathcal{B}$, then $B\notin \mathcal C_{\epsilon}(X)$, for some $B\in\mathcal{B}$.
	\end{center}
	Thus, by $(\dag)$, there exists $n_3>n_2$ such that, for each $i\in\{1,2\}$:
	\begin{center}
		if $\mathcal{B}\in [\mathcal C(X)]^{<3}$ and $[x_{n_i},x_{n_3}]=\bigcup\mathcal{B}$, then $B\notin \mathcal C_{\epsilon}(X)$, for some $B\in\mathcal{B}$.
	\end{center}	
	In this way, inductively, we construct a sequence $(x_{n_k})_k$ satisfying our requirements.
	
	\vspace{0.3cm}
	
	Finally, let $S=\{x_{n_k} : k\in\mathbb N\}\cup\{x_{\omega}\}$ and fix $L,T\in\mathcal{S}_c(S)$ such that $(L\setminus T) \cup (T\setminus L) $ is infinite. We prove that there is no  path joining $L$ and $T$ in $\mathcal{S}_c(X)$. Suppose that $L\setminus T$ is infinite and that there is a path $f\colon [0,1]\to \mathcal{S}_c(X)$ such that $f(0)=L$ and $f(1)=T$. By Theorem~\ref{teor:trayec-gamma-enX}, for each $x\in L$, there is a path $g_x\colon [0,1] \to X$ such that $g_x(0)=x$ and $g_x(t)\in f(t)$, for each $t\in [0,1]$. Furthermore, $\mathcal{P}=\{g_x[[0,1]] : x\in L\}$ has PSU, by Proposition~\ref{lemma2or3}. Let $\mathcal{R}=\{[x,g_x(1)] : x\in L\}$. Remark~\ref{rem:sref} implies that $\mathcal{R}$ is a strong refinement of $\mathcal{P}$. Thus, by Lemma~\ref{lem:srUAC}, the set $\mathcal{R}$ has PSU. Hence, there exists $n_0\in\mathbb N$ such that for each $x\in L$, there is $\mathcal{B}\in [\mathcal C_{\epsilon}(X)]^{<n_0+1}$, with $[x,g_x(1)]=\bigcup\mathcal{B}$.

Suppose that there exists an infinite set $M\subseteq L\setminus T$ such that $g_x(1)=x_{\omega}$ for each $x\in M$. Since $M\subseteq S$, $\{[x,x_1] : x\in M\}$ does not have PSU. Applying the same argument as in Claim 2, we have that $\{[x,x_{\omega}] : x\in M\}$ does not have PSU; but this contradicts the fact that $\{[x,x_{\omega}] : x\in M\}\subseteq\mathcal{R}$ and $\mathcal{R}$ has PSU (see Remark 3.11). Therefore, there exists $x_{n_k}\in L\setminus T$ such that $k>n_0$ and $g_{x_{n_k}}(1)=x_{n_l}\in T\setminus \{x_{\omega}\}$. Then	
	there is $\mathcal{B}\in [\mathcal C_{\epsilon}(X)]^{<n_0+1}\subseteq [\mathcal C_{\epsilon}(X)]^{<k+1}$ such that $[x_{n_k},x_{n_l}]=\bigcup\mathcal{B}$. But this contradicts Claim~4. Thus, there is no  path in $\mathcal S_c(X)$ joining $L$ and $T$. It is well known that the family $\mathcal{L}\subseteq [\mathbb N]^{\omega}$ such that $(L_1\setminus L_2) \cup (L_2\setminus L_1)$ is infinite, for each $L_1, L_2\in\mathcal{L}$, has cardinality $\mathfrak c$. Therefore, using Remark~\ref{rem:c}, we conclude that $\mathcal{S}_c(X)$ has exactly $\mathfrak c$ arc components.
\end{proof}

Our last result implies that  if $X$ is the Warsaw circle  or $X$ is any fan defined in section 6 of \cite{CharatonikFans}, then $\mathcal{S}_c(X)$ is not arcwise connected. In particular, Theorem~\ref{theoit875i} generalizes \cite[Example~2.8]{Garcia2015} and \cite[Example~3.7]{GarciaC}. 

  
\begin{corollary}  \label{corolit875iii}
	Let $X$ be a uniquely arcwise connected continuum. If $\mathcal{S}_c(X)$ is  arcwise connected, then   
	$X$ is  uniformly 	arcwise connected.
\end{corollary}

We note that the converse of Theorem~\ref{theoit875i} (or Corollary~\ref{corolit875iii}) is not true (see  \cite[Example~3.6]{GarciaC}).


\vspace{.3cm}

Next, we present another necessary condition for the path connectedness of $\mathcal{S}_c(X)$. To this end we introduce some notions.

Let $(C_n)_{n \in \mathbb{N}}$ be a sequence of nonempty subsets of a continuum $X$. The \textsl{limit superior} of the sequence, denoted by $\limsup C_n$, is the set of all points $x \in X$ such that there exist a sequence $(x_k)_{k \in \mathbb{N}}$ in $X$ converging to $x$ and a strictly increasing sequence $(n_k)_{k \in \mathbb{N}}$ in $\mathbb{N}$ satisfying that $x_k \in C_{n_k}$ for every $k \in \mathbb{N}$. The \textsl{limit inferior} of the sequence, denoted by $\liminf C_n$, is defined as the set of all points $x \in X$ such that there exists a sequence $(x_n)_{n \in \mathbb{N}}$ in $X$ converging to $x$ fulfilling that $x_n \in C_n$ for every $n \in \mathbb{N}$. If $\liminf C_n = \limsup C_n$, then the \textsl{limit} of the sequence is $\lim C_n = \liminf C_n$, and we say that the sequence $(C_n)_{n \in \mathbb{N}}$ converges to $\lim C_n$.

A proper, nonempty, closed subset $Y$ of a continuum $X$ is called:
\begin{itemize}
\item an $R^1$\textsl{-set} if there exist $U \in \tau_X$ containing $Y$ and sequences $(C_n^1)_{n \in \mathbb{N}}$ and $(C_n^2)_{n \in \mathbb{N}}$ of components of $U$ such that $Y = (\limsup C_n^1) \cap (\limsup C_n^2)$.
\item an $R^2$\textsl{-set} if there exist $U \in \tau_X$ containing $Y$ and sequences $(C_n^1)_{n \in \mathbb{N}}$ and $(C_n^2)_{n \in \mathbb{N}}$ of components of $U$ such that $Y = (\lim C_n^1) \cap (\lim C_n^2)$.
\item an $R^3$\textsl{-set} if there exist $U \in \tau_X$ containing $Y$ and a sequence $(C_n)_{n \in \mathbb{N}}$ of components of $U$ such that $Y = \liminf C_n$.
\end{itemize}

It is well known that an $R^i$-set contained in a continuum $X$ is an obstruction for the contractibility of its hyperspaces $\mathcal{K}(X)$ and $\mathcal{C}(X)$ \cite[Corollary~3.8]{Baik1997}. 
Theorem~\ref{theo:Ri} below implies that this also happens for the hyperspace $\mathcal{S}_c(X)$.

\begin{theorem} \label{theo-R3}
Let $X$ be a continuum. If $\mathcal{S}_c(X)$ is pathwise connected, then $X$ does not contain $R^3$-sets.
\end{theorem}
\begin{proof}
Seeking a contradiction suppose that there exists an $R^3$-set $Y$ contained in $X$. Let $U \in \tau_X$ and let $(C_n)_{n \in \mathbb{N}}$ be a sequence of components of $U$ such that $Y \subseteq U$ and $\liminf C_n = Y$. Fix $p \in Y$ and let $(p_n)_{n \in \mathbb{N}}$ be a sequence in $X$ such that $\lim p_n = p$ and $p_n \in C_n$ for every $n \in \mathbb{N}$. Set $S = \{p\} \cup \{p_n : n \in \mathbb{N} \} $. Notice that $S \in \mathcal{S}_c(X)$. 

Now, let $Q \in \mathcal{S}_c(X)$ be such that $Q \subseteq X \setminus Y$. Our assumption guarantees that there exists a path $\alpha : [0,1] \to \mathcal{S}_c(X)$ satisfying that $\alpha(0) = S$ and $\alpha(1) = Q$.  Define $T = \{ t \in [0,1] : \exists k \in \mathbb{N} \ (\alpha(t) \cap C_n \neq \emptyset \ \text{for all} \ n \geq k) \}$ and set $s = \sup T$.  \medskip

\noindent \textbf{Claim 1.} $\{ \lim \alpha(t) : t \in  T\} \subseteq Y$. \medskip

Let $t \in T$ be arbitrary. Then there exists $k \in \mathbb{N}$ such that $\alpha(t) \cap C_n \neq \emptyset$ for all $n \geq k$. Let $(x_n)_{n \in \mathbb{N}}$ be a sequence in $X$ such that $x_n \in \alpha(t) \cap C_n$ if $n \geq k$ and $x_n \in C_{n}$ if $n<k$. Thus, $\lim x_n = \lim \alpha(t) $. So, the last equality and the inclusion $\lim x_n \in \liminf C_n$ imply that $\lim \alpha(t) \in  Y$. This ends the proof of the claim. \medskip

\noindent \textbf{Claim 2.} $0 < s < 1$. \medskip

First, recall that $\alpha(0) = S \subseteq U$. From the continuity of $\alpha$, it follows the existence of $t_1 \in (0,1]$ satisfying that $\alpha[[0,t_1]] \subseteq \langle \{U\} \rangle$. We shall prove that $[0,t_1] \subseteq T$. 

Let $t \in [0,t_1]$ be arbitrary and let $n \in \mathbb{N}$ be arbitrary. Since $p_n \in \alpha(0)$, by Theorem~\ref{teor:trayec-gamma-enX}, there exists a mapping $\lambda _n : [0,t] \to X$ satisfying that $\lambda _n(0) = p_n$ and $\lambda _n(r) \in \alpha(r)$ for all $r \in [0,t]$. Observe that $\lambda _n[[0,t]]$ is a connected subspace of $\bigcup \alpha[[0,t]] \subseteq U$ and $p_n \in \lambda _n[[0,t]] \cap C_n$. Hence, $\lambda _n[[0,t]] \subseteq C_n$. Then $\lambda _n (t) \in C_n \cap \alpha _n(t)$. This implies that $t \in T$. We conclude that $0 < t_1 \leq \sup T$.

Second, the continuity of $\alpha$ and the fact that $\alpha(1) = Q \subseteq X \setminus Y$ ensures  the existence of $t_2 \in [0,1)$ fulfilling that $\alpha[[t_2,1]] \subseteq \langle \{X \setminus Y\} \rangle$. Let us argue that $[t_2,1] \cap T = \emptyset$. If $t$ were an element of $[t_2,1] \cap T$, the choice of $t_2$ would imply that $\alpha(t) \subseteq X \setminus Y$ and Claim~1 would guarantee that $\lim \alpha(t) \in Y$, a contradiction. Therefore, $[t_2,1] \cap T = \emptyset$ and we infer that $\sup T \leq t_2< 1$. Our claim is proved. \medskip

Finally, let $\mathcal{V} \in \mathfrak{C}(X)$ be such that: (i) $\alpha(s) \in \langle \mathcal{V} \rangle_c$ (see Proposition~\ref{pro:UCXbase}) and (ii) if $V \in \mathcal V$ is such that $V \cap Y \ne \emptyset$, then $V \subseteq U$.
From the fact that $\alpha$ is continuous, the choice of $s$ and Claim~2, it follows that there exist $s_1,s_2 \in [0,1]$ such that $s \in (s_1,s_2)$, $s_1 \in T$, $s_2 \notin T$ and $\alpha[[s_1,s_2]] \subseteq \langle \mathcal{V} \rangle$. Since $s_2 \notin T$, there exists an infinite subset $J$ of $\mathbb{N}$ such that   $\alpha(s_2) \cap C_n = \emptyset$ for each $n \in J$. On the other  hand, Claim~1 and the choice of $s_1$ guarantee that $\lim \alpha(s_1) \in Y$. Let $V \in \mathcal{V}$ be such that $\lim \alpha(s_1) \in V$. Then, by the choice of $\mathcal{V}$ we have that $V \subseteq U$. Claim~1 and the inclusion $s_1 \in T$ ensure the existence of $n \in J$ such that $\alpha(s_1) \cap C_n \cap V \neq \emptyset$. Fix $y \in \alpha(s_1) \cap C_n \cap V$. Now, by Theorem~\ref{teor:trayec-gamma-enX}, there exists a mapping $\rho : [s_1,s_2] \to X$ such that $\rho(s_1) = y$ and $\rho(r) \in \alpha(r)$ for each $r \in [s_1,s_2]$. Then $\rho[[s_1,s_2]]$ is a connected subspace of $\bigcup \alpha[[s_1,s_2]] \subseteq \bigcup \mathcal{V}$ and hence $\rho[[s_1,s_2]] \subseteq V \subseteq U$. This implies that $\rho[[s_1,s_2]] \subseteq C_n$, a contradiction. 
\end{proof}

The proof of our next result follows from the previous theorem, the fact that each $R^2$-set is an $R^3$-set \cite[Theorem~2.3, p.~310]{Baik1997} and the fact that each $R^1$-set contains an $R^3$-set \cite[Corollary~2.5, p.~310]{Baik1997}.

\begin{theorem}\label{theo:Ri}
Let $X$ be a continuum. If $\mathcal{S}_c(X)$ is pathwise connected, then $X$ does not contain $R^i$-sets for each $i \in \{1,2,3\}$.
\end{theorem}

We note that the converse of our last theorem fails (see \cite[Example~2.8]{Garcia2015}).
On the other hand, in \cite[Example~3.6]{GarciaC} the authors exhibited a  dendroid $X$ that contains an $R^i$-set and they proved  that $\mathcal S_c(X)$ has exactly $\mathfrak c$  arc components. 
Hence, the following question seems natural and interesting (cf. Theorem~\ref{theoit875i}).

\begin{question}
  Let $X$ be a continuum that contains an $R^i$-set for some $i \in \{1,2,3\}$. Is it true that $\mathcal S_c(X)$ has exactly $\mathfrak c$ arc components?
\end{question}  

In connection with this, in \cite[Question~3.8]{GarciaC} the authors asked if $\mathcal S_c(X)$ can have only a finite (but more than one) or even a countable number of arc components, for an arcwise connected space $X$.
Observe that Theorem~\ref{theoit875i} gives a partial answer to this question.

Further, as a consequence of Corollary~\ref{corolit875iii} and Theorem~\ref{theo:Ri} we have the following result and another interesting question.

\begin{corollary} \label{corol:ambas-cond}
    Let $X$ be a uniquely arcwise connected continuum (e.g. a dendroid). If $\mathcal S_c(X)$ is arcwise connected, then $X$ is uniformly arcwise connected and does not contain $R^i$-sets for any $i \in \{1,2,3\}$.
\end{corollary}    

\begin{question} \label{preg:uniqarcwiseconn}
   Is the converse of Corollary~\ref{corol:ambas-cond} true? In other words, do the uniform arc connectedness and the absence of $R^i$-sets characterize the arc connectedness of $\mathcal S_c(X)$, when $X$ is a uniquely arcwise connected continuum?
\end{question}   

In particular:

\begin{question}
  What about Question~\ref{preg:uniqarcwiseconn} when $X$ is a dendroid?
\end{question}  

\vspace{.1cm}

A nonempty subset $A$ of a space $X$ is said to be \textsl{homotopically fixed} if for each mapping $h : X \times [0,1] \to X$ such that $h(x,0) = x$ for every $x \in X$, we have that $h[A \times [0,1]] \subseteq A$. Since each $R^i$-set is homotopically fixed \cite[Theorem~3.2]{Baik1997}, it is natural to ask if the condition ``$X$ does not contain $R^i$-sets" in Theorem~\ref{theo:Ri} can be replaced with ``$X$ does not have homotopically fixed subsets". This is false: the continuum defined in \cite[Example~5]{CharatonikW1986} contains a homotopically fixed degenerate subset, does not contain $R^i$-sets and its hyperspace of nontrivial convergent sequences is pathwise connected, as we now show.


\begin{example}
There exists a continuum $X$ having a homotopically fixed  subset, such that $\mathcal{S}_c(X)$ is pathwise connected. In particular, $X$ does not contain $R^i$-sets.

We recall the construction of $X$ made in \cite[Example~5]{CharatonikW1986} for the sake of completeness. In the Euclidean plane, let $a = (0,1)$, $b = (0,0)$, $c = (0,-1)$ and for each $n,m \in \mathbb{N}$ let $$a_n = (2^{-3n},1), \ b_n = (2^{-3n},0), \ c_n = (2^{-(3n+1)},-1), \ d_n = (2^{-(3n-1)},-1),$$ $$b_{n,m} = (2^{-3n}(1-2^{-(m+3)}),0), \ b'_{n,m} = (2^{-3n}(1+2^{-(m+3)}),0),$$ $$c_{n,m} = (2^{-(3n+1)}(1-2^{-(m+3)}),-1) \ \text{and} \ d_{n,m} = (2^{-(3n-1)}(1+2^{-(m+3)}),-1).$$ Observe that $a_n \to a$, $b_n \to b$, $c_n \to c$ and $d_n \to c$ as $n \to \infty$; further, for any fixed $n \in \mathbb{N}$, we have that $b_{n,m} \to b_n$, $b'_{n,m} \to b_n$, $c_{n,m} \to c_n$ and $d_{n,m} \to d_n$ as $m \to \infty$.

For each $n \in \mathbb{N}$, set $Q_n = a_nb_nc_n \cup b_nd_n \cup \left( \bigcup\limits_{m\in \mathbb{N} } a_n b_{n,m} c_{n,m} \right) \cup \left( \bigcup\limits_{m \in \mathbb{N}} a_n b'_{n,m} d_{n,m} \right)$ (see Notation~\ref{broken}). Next, we define $Y = caa_1 \cup \bigcup\limits_{n \in \mathbb{N}} Q_n$ and let $Y'$ be the image of $Y$ under  the central symmetry with respect to the origin $b = (0,0)$. Such central symmetry will be denoted by $\pi$. Finally, set $X = Y \cup Y'$. Notice that $\{b\}$ is homotopically fixed. 

In order to prove that $\mathcal{S}_c(X)$ is pathwise connected, define 
    $$\mathcal{Y} = \{ S \cup L : S \in \mathcal{S}_c(Y \cup \pi[aa_1]) \cup \mathcal{S}_c(Y' \cup aa_1), 
      L \in [X]^{< \omega} \}.$$ 
    The path connectedness of $X$, the contractibility of $Y \cup \pi[aa_1]$ and $Y' \cup aa_1$, Theorem~\ref{teor:Xcontr-implica-Sc(X)-path-conn} and the fact that $\mathcal{S}_c(Y \cup \pi[aa_1]) \cap \mathcal{S}_c(Y'\cup aa_1) \neq \emptyset$ guarantee that $\mathcal{Y}$ is pathwise connected. 
    
    Now fix $S \in \mathcal{S}_c(X) \setminus \mathcal{Y}$. We will find a path in $\mathcal S_c(X)$ joining $S$ with an element of $\mathcal Y$. To this end observe that
   $$(\dag) \ \ |\{ n \in \mathbb{N} : S \cap Q_n \neq \emptyset \}| \ = \ \omega \ = \ |\{n \in \mathbb{N} : S  \cap \pi[Q_n] \neq \emptyset\}|;$$ 
 thus, $\lim S \in ca.$
Since both $X$ and $\mathcal Y$ are symmetric with respect to the origin, we may assume that 
     $$(\ddag) \ \ \lim S \in cb.$$
Let $\rho : X \to [-1,1]$ be the second coordinate projection.
\mbox{Define  $\lambda : S \setminus aa_1 \to X$ by}
    $$\lambda(x) = \left\{ \begin{array}{cl}
         x, & \text{if} \ x \in  \rho^{-1}[ \{ -1\}]  , \\    
         c, & \text{if} \ x \in ac  , \\
         c_n, & \text{if} \ x \in a_nb_nc_n  \ \text{for some} \ n \in \mathbb{N}, \\
         d_n, & \text{if} \ x \in b_nd_n \setminus \{b_n\} \ \text{for some} \ n \in \mathbb{N}, \\
         c_{n,m}, & \text{if} \ x \in a_nb_{n,m}c_{n,m} \ \text{for some} \ n,m \in \mathbb{N}, \\
         d_{n,m}, & \text{if} \ x \in a_nb'_{n,m}d_{n,m} \ \text{for some} \ n,m \in \mathbb{N}, \\
         \pi(a_n), & \text{if} \ y \in \pi[Q_n] \ \text{for some} \ n \in \mathbb{N}. \\
           \end{array} \right.$$ 
Conditions $(\dag)$ and $(\ddag)$ give $\lambda[S \setminus aa_1] \in  \mathcal{S}_c(Y \cup \pi[aa_1])$ and the continuity of $\lambda$.

Next we define a function $H:S \times [0,1] \to X$.
Let $(x, t) \in S \times [0,1]$. If $\rho(x) \notin \{-1,1\}$, let $H (x, t)$ be the unique point in the arc $[ x , \lambda (x)]$ whose second coordinate is $\rho(x)-t(1+ \rho(x))$ (see  Notation~\ref{unique-arcs}). If $\rho(x) \in \{-1,1\}$, simply define $H (x, t) = x$. Using $(\ddag)$ one can show that $H$  is continuous.
Hence, it follows from $(\dag)$ and Theorem~\ref{teor:func-ind-cont} that the function $\alpha : [0,1] \to \mathcal{S}_c(X)$ given by 
  \[\alpha(t) \ = \ \mathcal{K}(H) (S \times \{t\}) \ = \ H[S \times \{t\}]  \]
  is well defined and continuous. Thus, $\alpha$ is a path in $\mathcal S_c(X)$ from $S$ to some element of $\mathcal Y$.
This shows that $\mathcal{S}_c(X)$ is pathwise connected.

Finally, $X$ does not contain $R^i$-sets for any $i \in \{1,2,3\}$ by  Theorem~\ref{theo:Ri}.
\end{example}


\section{Local path connectedness}

In this section we characterize the local path connectedness of the hyperspace $\mathcal S_c(X)$ in terms of that of the space $X$ (Corollary~\ref{corol:equiv-locpathcon}). This allows us to provide another answer to Problem~\ref{probl:garcia}  (Corollary~\ref{corol:equiv-entre-path-conn y local-path-conn-XySc(X)}).

\begin{lemma} \label{lema:los-Q_U-2}
  Let $\mathcal U$ be a finite cellular family of a space $X$.  For each $U \in \mathcal U$ define $Q_U = \{ S \in \langle \mathcal U \rangle _c : \lim S \in U \}$. Let $U, V \in \mathcal U$ be such that both   $U$ and $V$ contain  arcs $A_U$ and $A_V$, respectively. If $x_\omega \in A_U$ and $y_\omega \in A_V$, then there exists a path $\gamma : [0,1] \to Q_U \cup Q_V$ such that $\lim \gamma (0) = x_\omega$ and $\lim \gamma(t) = y_\omega$ for each $t \in (0,1]$.
  In particular, $\gamma(0) \in Q_U$ and $\gamma \big( (0,1] \big) \subseteq Q_V$.
\end{lemma}
   \begin{proof} 
       Fix $S \in \mathcal S_c(A_U)$ such that $\lim S = x_\omega$ and let $S = \{ x_n : n \in   \omega +1 \} $ be an adequate enumeration of $S$. For each $n \in \mathbb{N}$ denote by $\overline{x_n x_{n+1}}$ the subarc of $A_U$ whose end points are $x_n$ and $x_{n+1}$. Pick a path $\alpha _n : \textstyle  [\frac{1}{n+1}, \frac{1}{n}] \to \overline{x_n x_{n+1}}$ such that $\alpha_n( \frac{1}{n} ) = x_n$ and $\alpha _n(\frac{1}{n+1}) = x_{n+1}$. 
Similarly,  we may take $ \{ y_n : n \in    \omega +1 \} \in \mathcal S_c(A_V)$ and for each $n \in \omega$ denote by $\overline{y_n y_{n+1}}$ the subarc of $A_V$ whose end points are $y_n$ and $y_{n+1}$. Also choose a path $\beta _n : \textstyle  [\frac{1}{n+1}, \frac{1}{n}] \to \overline{y_n y_{n+1}}$ such that $\beta_n( \frac{1}{n} ) = y_n$ and $\beta _n (\frac{1}{n+1}) = y_{n+1}$.

For each $W \in \mathcal U \setminus \{U,V\}$ fix a point $x_W \in W$ and set $Z = \{x_W: W \notin \{U,V\} \}$. Define $\gamma :[0,1] \to Q_U \cup Q_V$ by
     \begin{equation*}
         \gamma(t) = \left\{ \begin{array}{lll}
            \{x_m : m \le n \} \cup \{ \alpha_n(t)\} & & \\
\ \ \ \ \ \ \ \ \ \ \ \ \ \ \ \ \ \ \            \cup \{y_m : m \ge n+1 \}   \cup \{ \beta_n(t) \} \cup Z,  & \textrm{if} & t \in [\frac{1}{n+1},\frac{1}{n}]; \\
            \{ x_m : m \in  \omega +1 \} \cup \{y_\omega\} \cup Z, & \textrm{if} & t =0.
          \end{array} \right.
       \end{equation*}   
Then $\gamma$ is well defined and using  \cite[Theorem~18.3, p.~108]{Munkres} and Lemma~\ref{lema:union-finita-cont} one can show that it is continuous in $(0,1]$. In order to check that $\gamma$  is continuous at $t=0$, fix a finite cellular family $\mathcal W$ such that $\{ x_m : m \in  \omega +1 \} \cup \{y_\omega\} \cup Z \in \langle \mathcal W \rangle _c $. Let $W_0 , W_1 \in \mathcal W$ such that $x_\omega \in W_0$ and $y_\omega \in W_1$; further, choose $N \in \mathbb{N}$ satisfying that $\bigcup _{n \ge N} \overline{x_n x_{n+1}} \subseteq W_0$ and $\bigcup _{n \ge N} \overline{y_n y_{n+1}} \subseteq W_1$. Hence, if $t \in [0, \frac{1}{N})$ a routine argument shows that $\gamma(t) \in \langle \mathcal W \rangle _c$ and, thus, $\gamma$ is continuous. Finally, it is immediate to see that $\lim \gamma (0) = x_\omega$ and $\lim \gamma(t) = y_\omega$ for each $t \in (0,1]$.
   \end{proof}

\begin{theorem} \label{teor:local-con-lims-difs}
    Let $X$ be a space  and assume that $\mathcal U$ is a finite cellular family whose elements are pathwise connected. 
  Let $x  ,y  \in X$ be such that $X$ is first countable  and locally pathwise connected  at both $x $ and $y $.  If $S_x,S_y \in \mathcal \langle \mathcal U \rangle _c$ are such that $\lim S_x = x $ and $\lim S_y = y $, then  there exists a path in $\langle \mathcal U \rangle _c$ from $S_x$ to $S_y$.
\end{theorem}
   \begin{proof}
      Let $U_x, U_y \in \mathcal U$ be such that $x  \in U_x$ and $y  \in U_y$. We consider two cases. \medskip
      
     \noindent \textbf{Case 1.} $U_x = U_y$. \medskip

    Fix  $U \in \mathcal U \setminus \{U_x\}$. In this case the cellularity of $\mathcal U$ implies that both $S_x \cap U$ and $S_y \cap U$ belong to $\mathcal  F(U)$. Hence, by Remark~\ref{remark:path-conn}  we may take a path $\varphi_{U}: [0, 1] \to \mathcal F(U)$ such that $\varphi_{U}(0) = S_x \cap U$ and $\varphi_{U}(1) = S_y \cap U$.
     
Notice that  $S_x \cap U_x \in \mathcal S_c( U_x)$.
 Let $A \subseteq U_x$ be an arc that contains both $x$ and $y $ (Theorem~\ref{lema:path-con-implica-arco-con}) and let $A_x , A_y \in \mathcal S_c(A)$ be such that $\lim A_x = x$ and $\lim A_y = y $. According to Corollary~\ref{corol:loc-path-con-1onum}, the sequences $S_x \cap U_x$ and $A_x$ can be joined with a path in $\mathcal S_c( U_x)$. Similarly, $S_y \cap U_x$ and $A_y$ can be joined with a path in $\mathcal S_c( U_x)$. Moreover, by \cite[Corollary~1.5]{Garcia2015}, $A_x$ and $A_y$ can be joined with a path in $\mathcal S_c(A) \subseteq \mathcal S_c( U_x)$. Thus,  there exists a path $\varphi_{U_x} : [0,1] \to \mathcal S_c( U_x)$ such that $\varphi_{U_x}(0) = S_x \cap U_x$ and $\varphi_{U_x}(1) = S_y \cap U_x$. Finally define $\alpha : [0,1] \to \langle \mathcal U \rangle _c$ by $\alpha (t) =   \bigcup \{ \varphi_U (t) : U \in \mathcal U \}$. Lemma~\ref{lema:union-finita-cont} guarantees that  $\alpha$ is a path in $\langle \mathcal U \rangle _c$ joining $S_x$ and  $S_y$. \medskip


      \noindent \textbf{Case 2.} $U_x \ne U_y$. \medskip
      
      Using Theorem~\ref{lema:path-con-implica-arco-con} and Lemma~\ref{lema:los-Q_U-2} one can show that there exists a path $\gamma : [0,1] \to \langle \mathcal U \rangle _c$ such that $\lim \gamma (0) = x \in U_x$ and $\lim \gamma (1) = y \in U_y$. 
By Case~1 there exist two paths in $\langle \mathcal U \rangle _c$, one from $S_x$ to $\gamma (0)$ and the other from $S_y$ to $\gamma(1)$. The result follows.
    \end{proof}

\begin{corollary} \label{corol:Xfirst-count-loc-pathcon-implicaXtb}
    Let $X$ be a first countable space. If $X$ is  locally pathwise connected, then  so  is $\mathcal S _c(X)$.  
\end{corollary}
  \begin{proof}
          Fix $S \in \mathcal S _c(X)$ and a finite cellular family $\mathcal U$ such that $S \in \langle \mathcal U \rangle _c$; since $X$ is locally pathwise connected, by Lemma~\ref{lema:LAC-equiv} we may assume that the elements of $\mathcal U$ are pathwise connected. Applying Theorem~\ref{teor:local-con-lims-difs} we obtain that $\langle \mathcal U \rangle _c$ is pathwise connected.
  \end{proof}

\begin{corollary} \label{corol:Xfirst-count-path-loc-path}
    Let $X$ be a first countable space. If $X$ is pathwise connected and locally pathwise connected, then  $\mathcal S _c(X)$ is pathwise connected and locally pathwise connected.
\end{corollary}
  \begin{proof}
     Apply Theorem~\ref{teor:local-con-lims-difs} to $\mathcal U = \{X \}$ to obtain that $\mathcal S _c(X)$ is  pathwise connected.
   \end{proof}

Since locally connected continua are pathwise connected and locally pathwise connected \cite[8.23 and 8.25, pp.~130--131]{nadler-cont}, we obtain the following result.   

\begin{corollary}
  If $X$ is a locally connected continuum, then $\mathcal \mathcal{S}_c(X)$ is pathwise connected and locally pathwise connected.
\end{corollary}

The assumption on first countability in Corollary~\ref{corol:Xfirst-count-path-loc-path} is restrictive, as our next example shows.

\begin{example}
  Given a cardinal $\kappa$,  endowed  with the discrete topology, consider the topological product $X=[0,1]\times\kappa$ and set $F=\{0\}\times\kappa$. The {\sl hedgehog of $\kappa$ spines}, denoted by $J(\kappa)$, is the quotient $X/F$, i.e., the topological space which results of collapsing $F$ to a single point. 
  
  Note that $J(\omega)$ is   pathwise connected and locally pathwise connected, but   it is not first countable. Nevertheless  $\mathcal S_c( J(\omega) )$ is pathwise connected, as we now show.
  Given $S \in \mathcal S_c( J(\omega) )$, observe that there exists $n \in \omega$ such that $S \subseteq J(n)$. Hence, if $S ,R \in \mathcal S_c( J(\omega) )$,  then $S ,R \subseteq J(m)$ for some  $m \in \omega$. Since $\mathcal S_c( J(m) )$ is pathwise connected \cite[Lemma~1.8]{Garcia2015}, then so is $\mathcal S_c( J(\omega) )$.
\end{example}

\vspace{.3cm}

Corollary~\ref{corol:Xfirst-count-path-loc-path} states that the path connectedness, together with  the local path connectedness of a space $X$, is a sufficient condition to obtain the path connectedness of $\mathcal \mathcal{S}_c(X)$ -among first countable spaces-. Nevertheless, local path connectedness is not a necessary condition for the path connectedness of $\mathcal \mathcal{S}_c(X)$,  as the cone over the Cantor set shows (see Theorem~\ref{teor:Xcontr-implica-Sc(X)-path-conn}). The path connectedness of $X$ is however necessary  (Theorem~\ref{teor:Sc(X)-path-conn-implicaXtb}). Hence, the  problem of
finding necessary and sufficient conditions on a space $X$ for $\mathcal \mathcal{S}_c(X)$ to be pathwise connected 
 seems very interesting and remains unsolved.

\begin{lemma} \label{lema:loc-path-con-implicaarco}
  Let $X$ be   a  space. If $\mathcal \mathcal{S}_c(X)$ is  locally pathwise connected at some element, then $X$ contains an arc.
\end{lemma}
  \begin{proof}
     By assumption, there exists  a pathwise connected neighborhood $\mathcal V$ of $S$ for some $S \in \mathcal \mathcal{S}_c(X)$; hence we may take a finite cellular family $\mathcal W$ such that $S \in \langle \mathcal W \rangle _c \subseteq \mathcal V$. Set $p = \lim S$ and let $W$ be the element of $ \mathcal W$ that contains $p$. Take $z \in (S \cap W) \setminus \{p\}$ and define $R = S \setminus \{z\}$. Since $S,R \in \langle \mathcal W \rangle _c \subseteq \mathcal V$ there exists a path $\alpha: [0,1] \to \mathcal V$ such that $\alpha (0) = S$ and $\alpha (1) = R$. According to Theorem~\ref{teor:trayec-gamma-enX} there exists a path $\gamma: [0,1] \to X$ such that $\gamma(0) = z$ and $\gamma(1) \in R \subseteq X \setminus \{z\} $. The result follows from Theorem~\ref{lema:path-con-implica-arco-con} applied to $\gamma[ [0,1]]$.
  \end{proof}

\begin{theorem} \label{teor:Sc(X)locpathconn-implicaXlocpathconcik}
  Let $X$ be   a space.  If $ \mathcal \mathcal{S}_c(X)$ is nonempty and locally pathwise connected, then so is $X$.  
\end{theorem}
   \begin{proof}
       Fix $q \in X$ and an open subset $W$ of $X$ that contains $q$.      
       Since $X$ contains an arc (Lemma~\ref{lema:loc-path-con-implicaarco}), we may take   $S \in \mathcal \mathcal{S}_c(X)$ such that $q \in S \setminus \{\lim S\}$.  Let $U_1$ and $U_2$ be disjoint open sets such that $q \in U_1 \subseteq W$ and $S \setminus \{q\} \subseteq U_2$.
       
       Since $S \in \langle \{U_1, U_2\} \rangle _c$,  there exists a pathwise connected  set $\mathcal V$  such that $S \in \int(\mathcal V )\subseteq \mathcal V \subseteq \langle \{U_1, U_2\} \rangle _c$.      It follows from Lemma~\ref{lema:union-da-abierto}   that $U _1 \cap (\bigcup \int(\mathcal V))$ is an open subset of $X$ and that  $q \in U _1 \cap (\bigcup \int(\mathcal V)) \subseteq  U _1 \cap (\bigcup \mathcal V) \subseteq W$. 
    
    It suffices to prove that each point of $U _1 \cap (\bigcup \mathcal V)$ may be connected with $q$ by a path contained in $U _1 \cap (\bigcup \mathcal V)$. To this end fix $p \in U _1 \cap (\bigcup \mathcal V)$ and pick $S_p \in   \mathcal V  $  such that $p  \in S_p $. Also, let $\alpha : [0,1] \to \mathcal V$ be a path such that $\alpha(0) =S_p$ and $\alpha(1) =S$. According to Theorem~\ref{teor:trayec-gamma-enX} there exists a path $\gamma: [0,1] \to \bigcup \alpha \big( [0,1]\big)$ such that $\gamma(0) =p$ and $\gamma(1) \in S$.
       Observe that  $ \gamma \big( [0,1]\big) \subseteq \bigcup \mathcal V \subseteq  U_1 \cup U_2$. The fact that $U_1 \cap U_2 = \emptyset$ and the connectedness of $\gamma \big( [0,1]\big)$ imply that $  \gamma \big( [0,1]\big) \subseteq U _1 \cap (\bigcup \mathcal V)$. Finally, since $\gamma(1) \in S\cap U _1$ then $\gamma(1) = q$. Therefore,  $U _1 \cap (\bigcup \mathcal V)$ is pathwise connected. 
   \end{proof}

Our next two characterizations follow from Corollary~\ref{corol:Xfirst-count-loc-pathcon-implicaXtb}, Theorem~\ref{teor:Sc(X)locpathconn-implicaXlocpathconcik}, 
Corollary~\ref{corol:Xfirst-count-path-loc-path} and Theorem~\ref{teor:Sc(X)-path-conn-implicaXtb}.

\begin{corollary} \label{corol:equiv-locpathcon}
   The following conditions are equivalent for a first countable space $X$ with nonempty hyperspace $\mathcal{S}_c(X)$:
   
   (i) $X$ is locally pathwise connected;
   
   (ii) $\mathcal \mathcal{S}_c(X)$ is  locally pathwise connected.
\end{corollary}

\begin{corollary} \label{corol:equiv-entre-path-conn y local-path-conn-XySc(X)}
   The following conditions are equivalent for a  first countable space $X$,  with nonempty hyperspace $\mathcal{S}_c(X)$:
   
   (i) $X$ is pathwise connected and  locally pathwise connected;
   
   (ii) $\mathcal \mathcal{S}_c(X)$ is pathwise connected and  locally pathwise connected.
\end{corollary}

\begin{question}
   Can the assumption on first countability be removed in  Corollary~\ref{corol:equiv-locpathcon} or Corollary~\ref{corol:equiv-entre-path-conn y local-path-conn-XySc(X)}?
\end{question}


\section{Contractibility}

  In this section we prove that $\mathcal S_c(X)$ is contractible, whenever $X$ is a nondegenerate connected subspace of a tree (Definition~\ref{defin:tree}).

In the proof of our next result we use the following notation: for a subset $S$ of $\mathbb R$ and a real number $t$ we denote by $tS$ the set $\{tx : x \in S\}$.

\begin{theorem} \label{teor:Sc(arco)-contr}
   If   $X$ is a space of the form $(a,b), (a,b]$ or $[a,b]$, then $\mathcal S_c ( X)$ is contractible.
\end{theorem}
      \begin{proof}
         Without loss of generality (and for the sake of simplicity) we will assume that $-2 < a < 0 < 1 < b <2$. 
      
      \noindent \textbf{Claim.} The set $\mathcal X _1= \{ S \in \mathcal S_c( X) : 1 \in S\}$ is contractible.
      
\noindent   Consider the function $H: \mathcal X _1 \times [0,1] \to \mathcal X_1$ given by 
      \begin{equation*}  
          H(S,t) = \left\{ \begin{array}{lll}
               tS \cup \big\{ \frac{1}{2^{k-1}} : k \in \{1, \dots , n\} \big\},  & \textrm{if} & t \in 
                               [\frac{1}{2^n}, \frac{1}{2^{n-1}}] \textrm{ and } n \in \mathbb N; \\
               \{0\} \cup \big\{ \frac{1}{2^{k-1}} : k \in \mathbb N \big\},      & \textrm{if} & t =0.
          \end{array} \right.
       \end{equation*}
The fact that $H$ is well defined follows from recalling that $1 \in S$ for all $S \in \mathcal X _1$ and from noting that  $\frac{1}{2^n} \in \frac{1}{2^n}S$ for each $n \in \mathbb{N}$. Moreover, one can easily show that $H$ is continuous in $\mathcal X _1 \times (0,1]$. 

In order to prove that $H$ is continuous at the elements of $\mathcal X _1 \times \{0\}$, fix $S \in \mathcal X _1$ and let $\mathcal U$ be a finite cellular family such that $\{0\} \cup \big\{ \frac{1}{2^{k-1}} : k \in \mathbb N \big\} = H(S,0) \in \langle \mathcal U \rangle _c$. Let $U$ be the element of $\mathcal U$ that contains $0$ and pick $N \in \mathbb N$ such that $[\frac{-2}{2^N}, \frac{2}{2^N}] \subseteq U$.  
   Fix $(S_1,t) \in \mathcal X _1 \times \big[0, \frac{1}{2^{N}} \big)$.
   Observe that $tS_1 \subseteq \frac{1}{2^N}[a,b] \subseteq [\frac{-2}{2^N}, \frac{2}{2^N}] \subseteq U$. Thus, a straightforward argument gives $H(S_1, t) \in \langle \mathcal U \rangle _c$. This implies the continuity of $H$.  Finally, since $H(S,0) = \{0\} \cup \big\{ \frac{1}{2^{k-1}} : k \in \mathbb N \big\}$ and $H(S,1) =S$ for all $S \in \mathcal X _1$, the claim is proved.
 
 \vspace{.2cm}
 
Now define $M: \mathcal S_c( X) \to X$ by $M(S) = \max S$. Note that $M$ is continuous. 
Consider $G: \mathcal S_c( X) \times [0, \frac{1}{2}] \to \mathcal S_c( X)$ given by $G(S,t) = S \cup \{(1-2t)M(S) + 2t\}$, then $G$ is a homotopy such that  $G(S,0) = S$  and $G\big(S,\frac{1}{2} \big) \in \mathcal X _1$   for all $S \in \mathcal S_c(X)$.
Further, by the Claim we may take $S_0 \in \mathcal X _1$ and a homotopy $G_1: \mathcal X _1 \times [\frac{1}{2},1] \to \mathcal X _1$ such that $G_1 \big(S, \frac{1}{2} \big) = S$ and $G_1(S,1) = S_0$ for all $S \in \mathcal X _1$.     
  Define $G_2: \mathcal S_c(X) \times [0,1] \to \mathcal S_c(X)$ by    
      \begin{equation*}  
          G_2(S,t) = \left\{ \begin{array}{lll}
               G(S,t),     & \textrm{if} & t  \in [0, \frac{1}{2} ]; \\
               G_1 \big( G(S, \frac{1}{2}), t\big),     & \textrm{if} & t \in  [ \frac{1}{2} ,1].
          \end{array} \right.
       \end{equation*}
It follows easily that $G_2$ is a contraction in $\mathcal S_c(X)$.
   \end{proof}

  In the rest of this section we will use Notation~\ref{unique-arcs}.

\begin{definition} \label{defin:tree}
  By a \emph{tree} we mean a connected union of finitely many arcs, such  that it  does not contain simple closed curves.
\end{definition}

   Given a tree $X$ and $x \in X$, we say that $x$ is an \emph{end point} of $X$ provided that $x$ does not separate any arc in $X$ that contains it. Moreover, $x$ is a ramification point of $X$ if $x$ is the common end point of three arcs in $X$ that are otherwise disjoint.  The symbols $E(X)$ and $R(X)$ will denote the set of end points and the set of ramification points of $X$, respectively.  It is easy to see that both $E(X)$ and $R(X)$ are finite. 
   
Moreover, for each $p \in X$ we will consider the partial order given by:  $x <_p y$ provided that $x  \in [p,y] \setminus \{y\}$. Of course, $x \le _p y$ will mean that either $x <_p y$ or $x = y$.

\begin{theorem}
  If $X$ is a tree and $Y$ is a nondegenerate connected subset of $X$, then $\mathcal S_c(Y)$ is contractible.
\end{theorem}
   \begin{proof}
       Observe that $Y$ is  pathwise connected.
Fix $p \in (\int_X Y) \setminus R(X)$ and let $H : X \times [0,1] \to X$ be a homotopy such that: (i)~$H(x,0) = p$, (ii) $H(x,1) =x$ and (iii)  $H \big( H(x,s),t \big) = H(x, st)$ for each $x \in X$ and $s,t \in [0,1]$ (see \cite[Theorem~I-2-E, p.~549]{Fugate-Gordh-Lum} and  \cite[Remark before Theorem~II-3-B, p.~558]{Fugate-Gordh-Lum}). Using this last condition one can show that 
      \begin{equation} \label{eq:alfa-e-es-homeo-previo}
         \textrm{ $ H(x,s) <_p  H(x,t)$ whenever $s <t$ and $x \in X \setminus \{p\}$. }
      \end{equation}
   Given $x_1, x_2 \in X$ such that $x_1 <_p x_2$, note that there exists $s \in [0,1)$   such that $x_1 = H(x_2,s)$. If $t > 0$, it follows that $H(x_1,t) = H \big( H(x_2,s),t \big) = H(x_2 , st) <_p H(x_2,  t ) $; thus 
      \begin{equation} \label{eq:alfa-e-es-homeo}
         \textrm{ if $x_1 <_p x_2$, \ then }   H(x_1,t) <_p H(x_2,t) \ \textrm{ for all } t \in (0,1].
      \end{equation}
    Let $V$ be a neighborhood of $p$ in $X$ such that $V$ is an arc contained in $Y$ and let $t_0 >0$ be such that $H \big[ X \times [0,t_0] \big] \subseteq V$. According to Theorem~\ref{teor:Sc(arco)-contr}, we may take $S_0 \in \mathcal S_c(V )$ and a homotopy $G: \mathcal S_c(V) \times [0,t_0] \to \mathcal S_c(V)$ such that $G(S,0) = S_0$ and $G(S, t_0) =S$ for all $S \in \mathcal S_c(V)$. Use Theorem~\ref{teor:func-ind-cont} to define $\widehat G: \mathcal S_c(Y) \times [0,1] \to \mathcal S_c(Y)$ by 
      \begin{equation*}  
         \widehat G(S,t) = \left\{ \begin{array}{lll}
             \mathcal K(H) \big(S \times \{t\} \big),      & \textrm{if} & t \in [t_0, 1] ; \\
             G\big[ \mathcal K(H) \big(S \times \{t_0\} \big), t \big]  , & \textrm{if} & t \in[0, t_0].
          \end{array} \right.
       \end{equation*}
The facts that $X$ is a tree, that $Y$ is  pathwise connected and (\ref{eq:alfa-e-es-homeo-previo}) guarantee that $\mathcal K(H) \big(S \times \{t\} \big) \subseteq Y$ for all $(S,t) \in \mathcal S_c(Y) \times [0,1]$. Next, in order to see that $\widehat G$ is well defined fix $S \in \mathcal S_c(Y)$ and let $\{y_n : n \in \omega +1  \}$ be an adequate enumeration of $S$. 
Since $H(y_n,t)$  converges to $H(y_\omega,t)$ for each $t \in [0,1]$, we infer that
   \begin{equation} \label{eq:la-homot-llevaafinitoobiensuces}
     \mathcal K(H) \big(S \times \{t\} \big) =   H \big[S \times \{t\} \big] \in \mathcal F(Y) \cup \mathcal S_c(Y) \textrm{ for each } t \in [0,1].
    \end{equation}
The set $E(X)$ is finite, so we may fix $e \in E(X)$ such that $S \cap [p,e]$ is infinite. Then (\ref{eq:alfa-e-es-homeo}) implies that $H\big[(S \cap [p,e]) \times \{t\} \big]$ is infinite as well, for each $t >0$. Thus (\ref{eq:la-homot-llevaafinitoobiensuces}) yields $\mathcal K(H) \big(S \times \{t\} \big) \in \mathcal S_c(Y)$ for each $t >0$ and, therefore, $\widehat G$ is well defined. Finally, using Theorem~\ref{teor:func-ind-cont} it is easy to see that $\widehat G$ is a homotopy such that $\widehat G(S,0) = S_0$ and $\widehat G(S,1) = S$ for each $S \in \mathcal S_c(Y)$.
   \end{proof}  
   
The following question arises naturally.   
   
\begin{question}
  Let $X$ be a dendrite. Is $\mathcal S_c(X)$ contractible?
\end{question}

   Let $A_0 = [0,1] \times\{0\}$ and for each $n \in \mathbb{N}$ let $A_n$ be the segment in the plane that joins $(0,0)$ and $(1, \frac{1}{n})$. Define $X = \bigcup \{ A_n : n \in \omega \}$. Any space homeomorphic to $X$ is known as a \emph{harmonic fan}.

\begin{question}
  Let $X$ be a harmonic fan. Is $\mathcal S_c(X)$ contractible?
\end{question}

More generally:

\begin{question}
  Let $X$ be the cone over a space $Y$. Is $\mathcal S_c(X)$ contractible?
\end{question}

\begin{question}
  Let $X$ be a simple closed curve. Is $\mathcal S_c(X)$ contractible?
\end{question}

As a consequence of Corollary~\ref{corolit875iii}, Theorem~\ref{theo:Ri} and the fact that each contractible space is pathwise connected, we have the following two results (cf. \cite[Corollary~4]{CharatonikW1986}).

\begin{corollary}
Let $X$ be a uniquely pathwise connected continuum $X$. If $\mathcal{S}_c(X)$ is contractible, then $X$ is uniformly pathwise connected.
\end{corollary}

\begin{corollary} \label{corol:R3-set}
If a continuum $X$ contains an $R^i$-set for some $i \in \{1,2,3\}$, then $\mathcal{S}_c(X)$ is not contractible.
\end{corollary}


\section{Whitney blocks and Whitney levels}

Let $X$ be a compactum (i.e. a compact metrizable space). A Whitney map is a continuous function $w\colon \mathcal{CL}(X)\to [0, \infty)$ that satisfies the following conditions:
\begin{enumerate}
	\item for any $A,B\in \mathcal{CL}(X)$ such that $A\subsetneq B$, $w(A)<w(B)$;
	\item $w(A)=0$ if and only if $A\in \mathcal{F}_1(X)$.
\end{enumerate}
Note that the compactness of $\mathcal{CL}(X)$ allows us to assume that $w$ takes values in the interval $[0,1]$.

\begin{theorem} \cite[Theorem~13.4]{Illanes-Nadler}
	If $X$ is a compactum, then there is a Whitney map for $\mathcal{CL}(X)$.
\end{theorem}

Given a Whitney map $w\colon \mathcal{CL}(X)\to [0, 1]$, a \emph{Whitney level for $\mathcal S_c(X)$} is a set of the form $(w|_{\mathcal{S}_c(X)})^{-1}[\{t\}]$ and a \emph{Whitney block for $\mathcal S_c(X)$} is a set of the form $(w|_{\mathcal{S}_c(X)})^{-1}[[r,t]]$, where $0 \le r < t \le 1$.

In this section we study the connectedness of Whitney blocks and Whitney levels for $\mathcal{S}_c(X)$. 

\subsection{An auxiliary result.}

The short result in this subsection will be used in Proposition~\ref{prop:whit-omega+1}. As a by-product  we obtain a model for the hyperspace $\mathcal{S}_c(\omega+1)$.

\begin{proposition} \label{prop:modelo}
	Let $\mathcal{C}=\{0,1\}^{\omega}$ be the Cantor set. If $\mathcal{C}_{\omega}$ denotes the set $\{ (a_i)_{i\in\omega} \in \mathcal{C} : a_i=1\ \text{for infinitely many }i\}$ and if  $\varphi\colon\mathcal{C}_{\omega}\to \mathcal{S}_c(\omega+1)$ is defined, for each $(a_i)_{i\in\omega}\in \mathcal{C}_{\omega}$, by $\varphi((a_i)_{i\in\omega})=\{i : a_i=1\}\cup\{\omega\},$ then $\varphi$ is a homeomorphism.
\end{proposition}
  \begin{proof}
    	It is not difficult to see that $\varphi$ is a bijection. We show that $\varphi$ is continuous. 
	Fix $\mathcal{U} \in [\tau_{\omega +1}]^{<\omega}$.  Let $(a_i)_{i\in\omega}\in \varphi^{-1}[\langle \mathcal{U}\rangle_c]$, i.e.,   $\{i : a_i=1\}\cup\{\omega\}\in \langle \mathcal{U}\rangle_c$. 
	Let $U \in \mathcal U$ be such that $\omega \in U$. Hence, there exists  $n_0\in\omega$ such that $i\in U$ for each $i\geq n_0$. Define 
	 $$W=\{(b_i)_{i\in\omega}\in \mathcal{C}_{\omega} : b_i=a_i, \text{ for each }i\in\{0,...,n_0\}\}.$$ Notice  that $W$ is open in $\mathcal{C}_{\omega}$ and that $(a_i)_{i\in\omega}\in W  \subseteq \varphi^{-1}[\langle \mathcal{U}\rangle_c]$. Thus, $\varphi$ is continuous.
	
Now, we prove that $\varphi$ is open. For $n \in \mathbb N$ and $c_0,...,c_n \in \{0,1\}$ set
     $$U(c_0,...,c_n)=\{(b_i)_{i\in\omega}\in\mathcal{C}_{\omega} : b_i=c_i, \text{ for each }i\in\{0,...,n\}\}.$$
     Let $R\in \varphi[U(c_0,...,c_n)]$. We denote by $\{l_0,...,l_j\}=\{m : m\in R\cap\{0,...,n\}\}$. Let $V_i=\{l_i\}$ for each $i\in\{0,...,j\}$ and $V_{j+1}=\{m : m\geq n+1\}\cup\{\omega\}$, which are open subsets of $\omega +1$. 
      If $c_i = 0$ for each $i$, then we obtain directly that $R \in (V_{j+1})^+_c \subseteq \varphi[U(c_0,...,c_n)]$. Otherwise, let $\mathcal{V}=\{V_0,...,V_{j+1}\}$ and observe that $R\in \langle \mathcal{V}\rangle_c\subseteq \varphi[U(c_0,...,c_n)]$. Hence, $\varphi$ is open. 
  \end{proof}

\begin{remark}
Denote by $\mathbb I$ the space of the irrational numbers.
Recall that if $X$ is a nonempty Polish space (i.e. separable and completely metrizable), which is zero-dimensional and whose compact subspaces have all empty interior, then it is homeomorphic to $\mathbb I$ (\cite[Theorem~7.7]{Kechris}). Therefore,  Proposition~\ref{prop:modelo} implies that $\mathcal{S}_c(\omega+1)$ is homeomorphic to $\mathbb I$. This  was shown independently  in \cite[Theorem~2.2]{GarciaCP}, but the argument presented above is simpler.
\end{remark}

\subsection{Main results.}

\begin{theorem}\label{teo8493}
	If $X$ is a continuum and $w\colon \mathcal{CL}(X)\to [0,1]$ is a Whitney map, then $(w|_{\mathcal{S}_c(X)})^{-1}[\{t\}]\neq\emptyset$, for each $t\in (0,1)$.
\end{theorem}
\begin{proof}
	Since $X$ is a continuum, it is clear that $\mathcal{S}_c(X)$ is dense in $\mathcal{CL}(X)$. Furthermore, $\mathcal{S}_c(X)$ is connected, by \cite[Theorem~4.2]{MayaGScX}. Therefore, $w(\mathcal{S}_c(X))$ is  connected and dense in $[0,1]$.
\end{proof}

In the following proposition we prove that there exists a Whitney map  for \mbox{$\mathcal{S}_c(\omega+1)$} whose image is $(0,1]$ (cf. Theorem~\ref{teo8493}).

\begin{proposition} \label{prop:whit-omega+1}
	There exists a Whitney map $\mu\colon \mathcal{CL}(\omega+1)\to [0,1]$ such that $\mu |_{S_c(\omega+1)}\colon S_c(\omega+1)\to (0,1]$ is a bijection.
\end{proposition}

\begin{proof}		
	It is well known that $f\colon \mathcal{C}\to [0,1]$ defined by $$f((a_i)_{i\in\omega})=\sum_{i=1}^{\infty}\frac{a_i}{2^{i}},$$ is a continuous and onto map \cite[3.2.B, p.~146]{Engelking1989}. Furthermore, $f|_{\mathcal{C}_{\omega}}\colon \mathcal{C}_{\omega}\to (0,1]$ is a continuous bijection.
	
	Let $\varphi \colon\mathcal{C}_{\omega}\to \mathcal{S}_c(\omega+1)$
	be  as in Proposition~\ref{prop:modelo} and define $v\colon S_c(\omega+1)\to (0,1]$  by $v=f\circ \varphi ^{-1}$. Then $v$ is a continuous bijection. Observe that  $$\mathrm{cl}_{\mathcal{CL}(\omega+1)}(S_c(\omega+1))= S_c(\omega+1)\cup \{A\in \mathcal{F}(\omega+1) : \omega\in A\}.$$ 
	
	Let $\mathcal{H}= \mathcal{S}_c(\omega+1)\cup \{A\in \mathcal{F}(\omega+1) : \omega\in A\}$. Note that for each $A\in \mathcal{F}(\omega+1)$ such that $\omega\in A$, there exists a natural continuous extension of $v$ to $\mathcal{S}_c(\omega+1)\cup\{A\}$. Thus, there exists a continuous $v'\colon\mathcal{H}\to [0,1]$ such that $v'|_{\mathcal{S}_c(\omega+1)}=v$ by \cite[3.2.A(b), p.~146]{Engelking1989}. It is not difficult to show that $v'$ is a Whitney map. Therefore, by \cite[Theorem~16.10]{Illanes-Nadler}, there exists a Whitney map $\mu\colon \mathcal{CL}(\omega+1)\to [0,1]$ such that $\mu|_{S_c(\omega+1)}=v$.
\end{proof}

The proof of our next two results is very similar to that of \cite[Lemma~4.1 and Theorem~4.2]{MayaGScX}. We include the arguments here for the sake of completeness.

\begin{lemma}\label{lembloqueconexo}
 Let $X$ be a continuum, let $\mu\colon \mathcal{CL}(X)\to [0,1]$ be a Whitney map and let $t \in (0,1)$. If $S,P \in \mathcal{S}_c(X)$ are such that $S \in (\mu |_{\mathcal{S}_c(t)})^{-1}[[t,1)]$, $S \subseteq P$, and $P \setminus S$ is finite, then $S$ and $P$ belong to the same component of $(\mu |_{\mathcal{S}_c(t)})^{-1}[[t,1)]$.
\end{lemma}
\begin{proof}
Set $n = |P \setminus S|$ and let $\varphi : \mathcal{F}_n(X) \to \mathcal{S}_c(X)$ be defined by $\varphi(A) = S \cup A$ for each $A\in\mathcal{F}_n(X)$. The continuity of $\varphi$ is a consequence of Lemma~\ref{lema:union-finita-cont}. Since $S \subseteq \varphi(A)$ for each $A\in\mathcal{F}_n(X)$, the inclusion $\text{ran}(\varphi) \subseteq (\mu |_{\mathcal{S}_c(t)})^{-1}[[t,1)]$ holds. Now, by \cite[Theorem~4.10, p.~165]{Michael1951} $\mathcal{F}_n(X)$ is connected, hence so is ran$(\varphi)$. Finally, observe that $S,P\in \text{ran}(\varphi)$. This implies that $S$ and $P$ belong to the same component of $(\mu |_{\mathcal{S}_c(t)})^{-1}[[t,1)]$.
\end{proof}

\begin{theorem}\label{theobloqueconexo}
	Let $X$ be a continuum and let $\mu\colon \mathcal{CL}(X)\to [0,1]$ be a Whitney map. Then, $(\mu|_{\mathcal{S}_c(X)})^{-1}[[t,1]]$ is connected for each $t\in (0,1)$.
\end{theorem}
\begin{proof}
Let $Q,R \in (\mu|_{\mathcal{S}_c(X)})^{-1}[[t,1]]$. We denote by $\mathcal{C}$ the component of \break $(\mu|_{\mathcal{S}_c(X)})^{-1}[[t,1]]$ that contains $Q$. Fix $x \in X \setminus (R\cup Q)$. Observe that $R \cup \{x\}  \in (\mu|_{\mathcal{S}_c(X)})^{-1}[(t,1]].$ Let $\{r_i : i \in \omega +1\}$ and $\{q_i : i \in \omega +1\}$ be adequate enumerations of $R$ and $Q$, respectively.

For each $j\in\mathbb{N}$, let $F_j = \{r_\omega, x\} \cup \{ r_i : i \leq j \}$. Note that each $F_j\in\mathcal{CL}(X)$ and $\lim_{j\to\infty}F_j=R\cup\{x\}$ . This guarantees that there exists $j_0\in\mathbb{N}$ such that $\{ F_j : j \geq j_0 \} \subseteq \mu  ^{-1}[(t,1)]$.   

For each $k \in \mathbb{N}$ define  $Q_{k} = (Q \setminus \{ q_1, \ldots, q_k\}) \cup F_{j_0 +k}$ . Then  $Q_k \in (\mu|_{\mathcal{S}_c(X)})^{-1}[(t,1]]$. 
Using Lemma~\ref{lembloqueconexo} twice, we obtain that $Q \cup F_{j_0 +k} \in \mathcal C$ and, thus, that  $Q_k \in \mathcal{C}$. Hence, since $(Q_k)_{k \in \mathbb{N}}$ converges to $R \cup \{x,q_\omega\}$, we deduce that $R \cup \{x,q_\omega\} \in \mathcal{C}$. Therefore, the fact that $R \in \mathcal{C}$ follows from Lemma~\ref{lembloqueconexo}.
\end{proof}

\begin{theorem}\label{NivelNoConexo}
There exist a Whitney map $\mu\colon \mathcal{CL}(J)\to [0,\infty)$ for an arc $J$ and $t > 0$ such that  $(\mu|_{\mathcal{S}_c(J)})^{-1}[[0,t]]$ and $(\mu|_{\mathcal{S}_c(J)})^{-1}[\{t\}]$ are not connected.
\end{theorem}

\begin{proof}
Let $J = \{ (x,y) : \max \{ |x|, |y| \} = 1, y < 2|x| \}$. We shall prove that the Whitney map for $\mathcal{CL}(J)$ defined in \cite[13.5, p.~108]{Illanes-Nadler} satisfies the required properties. For the sake of completeness, we recall the construction of such Whitney map $\mu : \mathcal{CL}(J) \to [0, \infty)$.

First, for a subset $A$ of $J$, the symbol $\mathcal{F}_n(A)$ represents the set $[A]^{<n+1} \setminus \{\emptyset\}$.
Now, fix $A \in \mathcal{CL}(J)$. For any $n \in \mathbb{N} - \{1\}$, define $\lambda_n : \mathcal{F}_n(A) \to [0,\infty)$ as follows: for each $K \in \mathcal{F}_n(A)$, say $K=\{a_1, a_2, ..., a_n\}$, let 
      $$\lambda_n(K) =    \min \{d(a_i,a_j) : i \ne j \}$$
(in particular, if $|K |=1$ then $\lambda_n(K) =0$).      
Next, for each $n \in \mathbb{N} - \{1\}$, let $\mu_n(A) =  \text{lub} \{ \lambda_n(K) : K \in \mathcal{F}_n(A) \}$. Finally, let $$\mu(A) = \sum_{n=2}^{\infty} 2^{1-n} \cdot \mu_n(A).$$

Let us argue that $(\mu|_{\mathcal{S}_c(J)})^{-1}\left[[0,\frac{9}{8}]\right]$ is not connected by proving that there exists a nonempty proper closed open subset of $(\mu|_{\mathcal{S}_c(J)})^{-1}\left[[0,\frac{9}{8}]\right]$. Set $$U = \left\langle \left(-\frac{3}{4},-\frac{1}{2}\right] \times \{1\}, \left[\frac{1}{2},\frac{3}{4}\right) \times \{1\} \right\rangle \cap (\mu|_{\mathcal{S}_c(J)})^{-1}\left[\left[0,\frac{9}{8}\right]\right]$$ and $$V =  \left\langle \left[-\frac{5}{8},-\frac{1}{2}\right] \times \{1\}, \left[\frac{1}{2},\frac{5}{8}\right] \times \{1\} \right\rangle \cap (\mu|_{\mathcal{S}_c(J)})^{-1}\left[\left[0,\frac{9}{8}\right]\right].$$ Notice that $U$ is a nonempty proper open subset of $(\mu|_{\mathcal{S}_c(J)})^{-1}\left[[0,\frac{9}{8}]\right]$, $V$ is a nonempty proper closed subset of $(\mu|_{\mathcal{S}_c(J)})^{-1}\left[[0,\frac{9}{8}]\right]$ and $V \subseteq U$. To end the proof, it suffices to show that $U \subseteq V$. 

Let $S \in U$ be arbitrary. We will see that $S \cap ((-\frac{3}{4},-\frac{1}{2}] \times \{1\}) \subseteq [-\frac{5}{8},-\frac{1}{2}] \times \{1\}$ and that $S \cap ([\frac{1}{2},\frac{3}{4}) \times \{1\}) \subseteq [\frac{1}{2},\frac{5}{8}] \times \{1\}$. Suppose that there exists\break $x \in (S \cap ((-\frac{3}{4},-\frac{1}{2}] \times \{1\})) \setminus ([-\frac{5}{8},-\frac{1}{2}] \times \{1\})$. Take $y \in S \cap ([\frac{1}{2},\frac{3}{4}) \times \{1\})$. Then $\lambda_n(\{x,y\}) > \frac{9}{8}$ for each $n \geq 2$. Hence, $\mu_n(S) \geq \lambda_n(\{x,y\}) > \frac{9}{8}$, and so $\mu(S) > \frac{9}{8}$; a contradiction. Thus, $S \cap ((-\frac{3}{4},-\frac{1}{2}] \times \{1\}) \subseteq [-\frac{5}{8},-\frac{1}{2}] \times \{1\}$. Similarly, we can prove that $S \cap ([\frac{1}{2},\frac{3}{4}) \times \{1\}) \subseteq [\frac{1}{2},\frac{5}{8}] \times \{1\}$. Therefore, $S \in V$ and $U\subseteq V$.

The conclusion on $(\mu|_{\mathcal{S}_c(J)})^{-1}\left[\{\frac{9}{8}\}\right]$ follows from the fact that $U \cap (\mu|_{\mathcal{S}_c(J)})^{-1}\left[\{\frac{9}{8}\}\right]$ is a nonempty proper closed open subset of $(\mu|_{\mathcal{S}_c(J)})^{-1}\left[\{\frac{9}{8}\}\right]$.
\end{proof}

Recall that a continuous function $f:X \to Y$ is \emph{monotone} provided that its fibers are connected.

In our last result, the symbol $H(A,B)$ will denote the Hausdorff distance between $A$ and $B$ (\cite[Definition~2.1]{Illanes-Nadler}).

\begin{theorem}
Let $X$ be a continuum and let $\mu$ be a Whitney map for $\mathcal{CL}(X)$. If $(\mu|_{\mathcal{S}_c(X)})^{-1}(t)$ is connected and $\{ A \in \mu^{-1}(t) : \langle A \rangle_c \neq \emptyset \}$ is dense in $\mu^{-1}(t)$ for each $t \in (0,\mu(X))$, then $\mu$ is monotone.
\end{theorem}
\begin{proof}
Fix $t \in (0, \mu(X))$ be arbitrary.  First, we shall show that $\{ A \in \mu^{-1}(t) : \langle A \rangle_c \neq \emptyset \} \subseteq \text{cl}_{\mathcal{CL}(X)} (\mu|_{\mathcal{S}_c(X)})^{-1}(t)$.

Let $A \in \mu^{-1}(t)$ be such that $\langle A \rangle_c \neq \emptyset$ and let $\varepsilon > 0$ arbitrary. Take a component $C$ of $A$. Then $C \neq X$. So, there exists a subcontinuum $D$ of $X$ satisfying that $C \subsetneq D \subseteq N(\varepsilon,C)$. Choose $p \in D \setminus A$. Since $\mathcal F(A)$ is  dense in $\mathcal{CL}(A)$, there exists $B \in \mathcal F(A)$ such that $\mu(B \cup \{p\}) > t$ and $H(A,B) < \varepsilon$. Fix $S \in \langle A \rangle_c$. Set $Q = S \cup B$. Observe that $Q \in \langle A \rangle_c$ and $\mu(Q \cup \{p\}) > t$. Define the map $g : D \to \mathcal{S}_c(X)$ by $g(x) = Q \cup \{x\}$. Notice that $H(A,g(x)) < \varepsilon$ for each $x \in D$, $\mu(g(y)) \leq \mu(A) = t$ for each $y \in A$ and $\mu(g(p)) > t$. Thus, this and the connectedness of $D$ imply that there exists $z \in D$ satisfying that $\mu(g(z)) = t$. This proves that $A \in \text{cl}_{\mathcal{CL}(X)} (\mu|_{\mathcal{S}_c(X)})^{-1}(t)$.

Finally, our assumption that $\{ A \in \mu^{-1}(t) : \langle A \rangle_c \neq \emptyset \}$ is dense in $\mu^{-1}(t)$ guarantees that $\mu^{-1}(t) = \text{cl}_{\mathcal{CL}(X)} (\mu|_{\mathcal{S}_c(X)})^{-1}(t)$ and so $\mu^{-1}(t)$ is connected.
\end{proof}

\end{document}